%% file: higher__DISPERSIV_linear__test2.tex
\documentclass[10pt, reqno]{amsart}
\usepackage{graphicx} 
\usepackage{amsthm}
\usepackage{amsmath}
\usepackage{amssymb}
\usepackage{amsfonts}
\usepackage[left=2cm, right=2cm,top=2cm,bottom=2cm]{geometry}
\usepackage{bbm}
\usepackage{mathtools}
\usepackage[x11names]{xcolor}
\usepackage{hyperref}
\hypersetup{colorlinks=true,linkcolor=DarkOrchid4,citecolor=SpringGreen4, filecolor=magenta}
\usepackage{enumerate}

\usepackage{tikz}
\usepackage{tikzit}
\input{jitz.tikzstyles}
\usepackage{cleveref}
\usepackage{orcidlink}

\newtheorem{theorem}{Theorem}[section]
\newtheorem{proposition}{Proposition} [section]
\newtheorem{definition}{Definition}

\numberwithin{equation}{section}
\newtheorem{remark}{Remark}
\newtheorem{corollary}{Corollary}[section]

\title[Finite-dimensional controllability of KdV-type equation by linear test]{A linear test approach to global controllability of third- and fifth-order nonlinear dispersive equations}
\author[D. Mondal]{Debanjit Mondal$^{\orcidlink{0009-0002-8709-3774}}$}

\thanks{Department of Mathematics \& Statistics, Indian Institute of Science Education and Research Kolkata, Mohanpur, Nadia, \newline\indent West Bengal- 741246, India. Email address:  wrmarit@gmail.com}

\begin{document}
\begin{abstract}
	We investigate third- and fifth-order nonlinear dispersive equations of KdV type on the torus and establishes approximate controllability by a fixed four-dimensional control; rather than relying solely on the saturation machinery, the analysis exploits the finite-dimensional controllability of the inviscid Burgers equation linearized around a carefully constructed return trajectory, with the trajectory itself obtained from an observable family. This ``linear test" strategy, yields more information about the structure of the control than the standard approach. In particular, the constructed control is shown to depend continuously on the initial and target states, a property that is by no means automatic in nonlinear control problems, and to decompose as a bounded linear operator applied to the data plus a fixed remainder, with the operator part interestingly  independent of the order of dispersion. 
\end{abstract}

\allowdisplaybreaks
\maketitle
\section{Introduction}
\subsection{Model Description}

In this article, we study the dynamics of nonlinear dispersive waves described by Korteweg--de Vries (KdV)-type equations posed on the one-dimensional torus $\mathbb{T} := \mathbb{R}/2\pi\mathbb{Z}$. More precisely, we consider the evolution equation
\begin{align}\label{generalized_sysm}
	\partial_t u + (-1)^{j+1}\partial_x^{2j+1} u + \frac{1}{2}\partial_x(u^2) = 0,
	\qquad (t,x)\in(0,\infty)\times\mathbb{T},
\end{align}
where $j\in\{1,2\}$ and $u=u(t,x)$ is a real-valued function. This equations includes two classical dispersive models: for $j=1$, \eqref{generalized_sysm} reduces to the celebrated Korteweg--de Vries (KdV) equation, while for $j=2$ it corresponds to the Kawahara equation.

The KdV equation was derived by Korteweg and de Vries \cite{KdV_1895} in 1895 as a model for the propagation of surface water waves in a channel. Later, Kawahara \cite{kawahara_1972} in 1972 studied a fifth-order dispersive equation and, through numerical simulations, showed that it admits solitary wave solutions. As in the KdV case, these solutions can be either compressive or rarefactive, depending on the sign of the dispersive term.

\subsection{The control problem} We study the controllability of the KdV-type equation:  
\begin{align}\label{ctrl_prblm}  
	\begin{dcases}  
		\partial_t u + (-1)^{j + 1} \partial_x^{2j + 1} u + \frac{1}{2} \partial_x(u^2) = \eta(t, x), \quad (t, x) \in (0, T) \times \mathbb{T}, \\  
		u(0, x) = u^0 (x), \quad x \in \mathbb{T},  
	\end{dcases}  
\end{align}  
where $T>0$, $u^0$ denotes the initial data, and the external force $\eta$ serves as the control input.

To formalize, we introduce the notion of approximate controllability for the system.

\begin{definition}\label{approxi_cntrl}
	Let $H$ be a Hilbert space and $\mathcal{G}$ a vector space. The equation \eqref{ctrl_prblm} is said to be approximately controllable in $H$ with $\mathcal{G}$-valued controls if, for any given $T>0$, $\varepsilon>0$, and $u^0,u^1 \in H$, there exist a control $\eta$ taking values in $\mathcal{G}$ and a corresponding solution $u$ of \eqref{ctrl_prblm} such that
	\[
	\|u(T)-u^1\|_{H} \le \varepsilon.
	\]
\end{definition}

Compared to classical boundary and distributed control strategies, finite-dimensional controls have recently attracted significant attention. This work, we focus on the case where $\mathcal{G}$ is a finite-dimensional space.

Considering
\[
\mathcal{G} := \operatorname{span}\{\sin(x),\cos(x)\},
\]
for the KdV equation ($j=1$), \cite{Chen23} established approximate controllability, while for the Kawahara equation ($j=2$), approximate controllability was established in \cite{Ahamed_Mondal_25}, in both cases in the sense of Definition~\ref{approxi_cntrl}.

In this article, we establish approximate controllability of \eqref{ctrl_prblm} by exploiting \textit{the finite-dimensional controllability of the inviscid Burgers system, linearized around a suitably constructed return trajectory}. This approach gives some new information about the control structure. For this framework, we require the finite-dimensional control space to be
\begin{align}\label{defn_subspace_H}  
	\mathcal{H} = \operatorname{span} \{ \sin(x), \cos(x), \sin(2x), \cos(2x) \}.
\end{align}
To formulate our main result, we first introduce the relevant function spaces.

\subsection{Function spaces} Equations of the form \eqref{generalized_sysm} possess the conservation of mass, energy, and a Hamiltonian structure. These invariants play a central role in the qualitative and quantitative analysis of nonlinear dispersive equations, particularly in the study of well-posedness, stability, and long-time behavior of solutions.

In view of mass conservation, it is natural to restrict the analysis of \eqref{generalized_sysm} to the subspace of mean-zero functions. Accordingly, we consider initial data in the mean-zero Sobolev space
\[
H^s_0(\mathbb{T}) := \left\{ u \in H^s(\mathbb{T}) : [u] = \int_{\mathbb{T}} u(x) \, dx = 0 \right\},
\]
where \( H^s(\mathbb{T}) \) denotes the standard Sobolev space of order \( s \in \mathbb{N}\) on the torus \( \mathbb{T} \). We denote the \( L^2(\mathbb{T}) \)-norm by \( \| \cdot \| \), and the \( H^s \)-norm by \( \| \cdot \|_s \).

\subsection{Main result} 
We are now in a position to announce the main result of this article.
\begin{theorem}\label{main_thm}
	For any \( (j, s) \in \{1, 2\} \times \mathbb{N} \), the equation \eqref{ctrl_prblm} is approximately controllable in $H^{s+ 2j + 1}_0(\mathbb{T})$ at small time using a control that takes values in \( \mathcal{H} \). That is, for any $T > 0, $ any initial state \( u^0 \in H^{s+ 2j + 1}_0(\mathbb{T}) \) and target state \( u^1 \in H^{s+ 2j + 1}_0(\mathbb{T}) \), there exists a sufficiently small \( \tau > 0 \) and a control function \( \eta_{\tau, j} \in L^2((0, T\tau); \mathcal{H}) \) such that a solution \( u \) of \eqref{ctrl_prblm} satisfies  
	\begin{align}\label{lim_u_Ttau}  
		u_j(T \tau)\footnotemark \to u^1 \quad \text{in } H^s(\mathbb{T}) \text{ as } \tau \to 0^+.  
	\end{align}  
	Furthermore, the control \( \eta_{\tau, j} \) can be expressed as  
	\begin{align}\label{form_of_ctrl}  
		\eta_{\tau, j} = \mathcal{C}_{\tau}(u^0, u^1) + \kappa_{\tau, j},  
	\end{align}  
	where  
	\[
	\mathcal{C}_{\tau}: H^s_0(\mathbb{T}) \times H^s_0(\mathbb{T}) \to L^2((0, T\tau); \mathcal{H})  
	\]  
	is a bounded linear operator with a finite-dimensional range, and \( \kappa_{\tau, j} \in L^2((0, T\tau); \mathcal{H}) \). Importantly, both \( \mathcal{C}_{\tau} \) and \( \kappa_{\tau, j} \) are independent of the specific choices of \( u_0 \) and \( u_1 \). Limit \eqref{lim_u_Ttau} is uniform with respect to $u^0$ and $u^1$ in a bounded set of $H^{s+ 2j + 1}_0(\mathbb{T}).$
\end{theorem}
\footnotetext{The subscript in $u_j, \eta_j, \kappa_j$ indicates that the corresponding quantities $u, \eta,$ and $\kappa$ depending on $j$.}
As a direct consequence of the Main Theorem, we also establish approximate controllability in a given time.
\begin{corollary}\label{Coro_to_cnclude_main_result}  
	For any \( (j, s) \in \{1, 2\} \times \mathbb{N} \), the equation \eqref{ctrl_prblm} exhibits global approximate controllability in \( H^s_0(\mathbb{T}) \) at any given time \( T > 0 \). That is, for any \( \varepsilon > 0 \) and any initial and target states \( u^0, u^1 \in H^s_0(\mathbb{T}) \), there exists a control function \( \eta_j \in L^2((0, T); \mathcal{H}) \) such that the corresponding solution \( u \) of \eqref{ctrl_prblm}, satisfies  
	\[
	\|u_j(T) - u^1 \|_{s} \leq \varepsilon.
	\]  
\end{corollary}

\subsection{Main contributions of the present work} The global controllability result is of significant importance for nonlinear partial differential equations, since it imposes no restrictions on the amplitude of the initial and final states or on the control time $T$. In comparison with \cite{Chen23, Ahamed_Mondal_25}, we also establish global approximate controllability using finite-dimensional controls, where the dimension of the control space is independent of both the order of the dispersive term and the choice of initial and final data. This article provides some new information about the structure of the control. The control, given in the form \eqref{form_of_ctrl}, is uniform with respect to the choice of initial and target states. Moreover, the bounded operator part $C_{\tau}$ of the control is independent of the order of dispersion. In addition, the constructed control depends continuously on the initial and final states, a property that is not immediate in the context of nonlinear control problems. Such conclusions do not follow directly from the results in \cite{Chen23, Ahamed_Mondal_25}.

\subsection{Literature on dispersive equations}
\subsubsection{Well-posedness} Local well-posedness for a class of dispersive equations with more general nonlinearities was studied in \cite{Kenig_Ponce_gustavo_1994} using Bourgain spaces, introduced in \cite{Bourgain_1997,Bourgain_1993}. This approach was further developed in \cite{Kenig_Carlos_Ponce_1996} and \cite{Tao_2001}. The global well-posedness of \eqref{generalized_sysm} for $j=1$ and $j=2$ was established in \cite{Coliander_Keel_Staffilani_Tao_2003} and \cite{Kato_2012}, respectively.
\subsubsection{Controllability} Regarding control theory, when $j = 1,$ the system \eqref{ctrl_prblm} exhibits good control properties. The study of controllability for the KdV equation dates back to the late 1980s with the works \cite{Russell_1991} and \cite{Zhang_1990_thesis}. Since then, the controllability of the KdV equation has been extensively investigated in the literature; we briefly recall some of these results. For internal controllability of the KdV equation on a periodic domain, we refer to \cite{Russell_Zhang_1993, Russell_Zang_1996, Rosier_10}. In addition, \cite{Capistrano_Roberto_2015} studied the null controllability of the KdV equation on a finite domain. In both cases, the control region is an arbitrary open subdomain. Concerning boundary controllability, when the left control input is active, \cite{Rosier_2004} proved that the KdV equation is locally exactly trajectory controllable; this result was later improved in \cite{Glass_Guerrero_2009} in terms of regularity of the trajectory. If the first spatial derivative at the right endpoint is taken as the control, \cite{Rosier_1997} established local exact controllability of the KdV equation provided that the length of the spatial domain does not belong to a critical set. Subsequently, \cite{Coron_2004,Cerpa_2007, Cerpa_Crepeau_2009} obtained local exact controllability results on critical domains. More recently, local null-controllability of the nonlinear KdV equation with Dirichlet boundary conditions, using a Neumann boundary control on the right, was established in \cite{Coron_Koenig_Nguyen_24}, while local controllability of the KdV equation with a right Dirichlet control was studied in \cite{Nguyen_2025}. For a comprehensive overview of controllability results for the KdV equation, we refer to the survey by Cerpa \cite{Cerpa_2014} and the references therein.

Considering \( j = 2 \), the system \eqref{ctrl_prblm} corresponds to the so-called Kawahara equation. To the best of our knowledge, the control problem for this equation was first studied in \cite{Zhang_Zhao_2012, Zhang_Zhao_2015}, where the authors considered the Kawahara equation on the periodic domain \( \mathbb{T} \) with a distributed control. In \cite{Araruna_Capistrano_Doronin_2012}, the stabilization problem and the critical length phenomenon for the Kawahara equation were investigated.  Furthermore, local and global controllability results for a general higher-order KdV-type operator posed on the unit circle are presented in \cite{Roberto_Kwak_Leal_2022}. More recently, \cite{Araruna_Doronin_Rosier_2026} proved that local exact controllability of the Kawahara equation holds provided that the length of the spatial interval does not belong to a countable set of critical lengths.

\subsection{Methodology and related literature}
\subsubsection{Controllability of partial differential equations (PDEs) by finite-dimensional control force} The controllability of nonlinear PDEs under the action of an additive finite-dimensional forcing has attracted significant attention in recent years. This line of research was initiated in \cite{Agrachev_2005, Agrachev_2006}, where approximate controllability results for the two-dimensional Navier–Stokes and Euler equations on the torus were established. Since then, this framework has become widely known as the Agrachev–Sarychev approach, and their methodology has been successfully adapted to a broad class of nonlinear PDEs. It has been applied, for instance, to the three-dimensional Navier–Stokes equations \cite{Shirikyan_06, Shirikyan_07}, as well as to both compressible and incompressible Euler systems \cite{Nersisyan_2010, Nersisyan_2011}. Further developments include applications to the viscous Burgers equation \cite{Shirikyan_14, Shirikyan_18}, the Schrödinger equation \cite{Sarychev_12}, and, more recently, to the semiclassical cubic Schrödinger equation, where small-time approximate controllability of quantum density and momentum has been established \cite{Coron_Xiang_Zhang_2023}. The technique has also been extended to parabolic systems with polynomial nonlinearities \cite{Narsesyan_21}, through the analysis of suitable asymptotic limits. Using this asymptotic limit approach, global approximate controllability with finite-dimensional controls has been established for a variety of models, including the Camassa–Holm equation \cite{Chowdhury_Dutta_Mondal_25}, the Benjamin–Bona–Mahony equation \cite{Jellouli_23}, and the Kuramoto–Sivashinsky equation \cite{Peng_Gao_2022}.
\subsubsection{Return method} The return method, originally introduced in \cite{Coron_return_stabilization_1992} for stabilization problems, has since become a powerful tool in the study of control theory for nonlinear PDEs. Its first major application to global exact boundary controllability was demonstrated in \cite{Coron_JMPA_1996} for the two-dimensional incompressible Euler system. This method was further extended to various fluid models: in \cite{Coron_Fursikov_1996}, it was employed to establish global exact controllability to trajectories for the 2D Navier–Stokes system without boundary, and in \cite{Coron_ESIAM_1996}, to prove approximate controllability of the 2D Navier–Stokes system with Navier slip boundary conditions, using controls localized in physical space or along the boundary. Subsequent developments include its application to the 3D Euler equations in \cite{Glass_2000}, and to the 3D Boussinesq system in \cite{Fursikov_Imanuvilov_1999}, where global exact controllability to trajectories was addressed. Years later, this method was adapted to both nonviscous and viscous Burgers-type equations \cite{Chapouly_Burger_2009}, as well as to the global controllability of a nonlinear KdV equation \cite{Chapouly_KdV_2009}, where the author used the controllability of the nonviscous Burgers equation to obtain controllability results for the KdV equation. For a overview of the return method in control theory, we refer \cite[Chapter 6]{Coron_book_2007} and the references therein.
\subsubsection{Our approach} The methodology of this work is primarily inspired by \cite{Nersesyan_2021_Sicon}, where approximate controllability of the three-dimensional incompressible Navier--Stokes equations was established using a combination of the return method and the notion of observable families. We also draw motivation from \cite{Chapouly_KdV_2009}, where global controllability of a nonlinear KdV equation was obtained via the controllability of the inviscid Burgers equation. In the present work, we adapt the approach of \cite{Nersesyan_2021_Sicon} by combining it with the return method developed in \cite{Chapouly_KdV_2009} within a finite-dimensional control framework. To the best of our knowledge, this is the first application of the Agrachev--Sarychev framework, combined with the return method, to this class of nonlinear dispersive equations.

The main idea of the proof can be outlined as follows. We start with the inviscid Burgers equation \eqref{burger_equn}. A suitable finite-dimensional return trajectory $w(t)$ (cf.~\eqref{expression_w}) is constructed, which starts and ends at the origin over a fixed time interval $[0,T]$, and is associated with a control taking values in a finite-dimensional subspace. The equation is then linearized around this trajectory, and $v(t)$ denotes the solution to the resulting linearized system \eqref{linearized_Burger}, driven by another control that also takes values in the same finite-dimensional space. During the construction of $w(t)$, the time-dependent components (see \eqref{expression_w}) are chosen to form an observable family (see Definition~\ref{defn_Observable_family}). This observability ensures that the time-dependent linearized system \eqref{linearized_Burger} is approximately controllable using controls from the same finite-dimensional subspace. A control is then constructed such that, for each $j \in \{1,2\}$, the corresponding solutions $u_j(T\tau)$ and $v(T)$, issued from the same initial data, remain close in the $H^s$-norm as $\tau \to 0^+$. The approximate controllability of the linearized system, together with the notion of an approximate right inverse, allows one to conclude the limit relation \eqref{lim_u_Ttau}, with the control explicitly constructed in the form \eqref{form_of_ctrl}. Finally, the operator $\mathcal{C}_\tau$ serves as an approximate right inverse to the resolving operator of the linearized Burgers system, and $\kappa_{\tau,j}$ is explicitly expressed in terms of the trajectory $w$ and the corresponding control $\xi$.

\subsection{Structure of the paper} In Section~\ref{sec_2}, we begin by recalling the well-posedness results for the KdV-type equation and introduce our idea of the linear test. In Section~\ref{sec_3}, we study the controllability of a linear transport equation with a specific time-space dependent coefficient using finite-dimensional controls. Section~\ref{sec_4} is devoted to establishing the closeness of solutions to the linearized equation and the KdV-type equation. Finally, in Section~\ref{section-aaprox_control}, we complete the proof of the Theorem~\ref{main_thm} and Corollary~\ref{Coro_to_cnclude_main_result}.

\subsection*{Acknowledgments}  
The author gratefully acknowledges Prof. Vahagn Nersesyan for clarifying certain doubts related to the article \cite{Nersesyan_2021_Sicon}. Thanks are also due to the audience of the poster session at the conference \textit{Control of PDEs and Related Topics}, Toulouse, 2025, for their valuable feedback. The author sincerely thanks Dr. Shirshendu Chowdhury, Dr. Rajib Dutta and Dr. Subrata Majumdar for helpful discussions. This work was supported by the Integrated PhD Fellowship at IISER Kolkata, India.

\section{Existence and properties}\label{sec_2}
We rewrite the system \eqref{ctrl_prblm} as  
\begin{align}\label{diff_form_ctrl_prblm}  
	\begin{dcases}  
		\partial_t u + \mathcal{L}_j u + \mathcal{B}(u) = f, \quad (t, x) \in (0, T) \times \mathbb{T}, \\  
		u(0, x) = u^0 (x), \quad x \in \mathbb{T},  
	\end{dcases}  
\end{align}  
where the linear operator $\mathcal{L}_j$ is given by  
\[
\mathcal{L}_j := (-1)^{j + 1} \partial_x^{2j + 1}, \quad \text{for all } j \in \{1, 2\},  
\]  
and the nonlinear operator $\mathcal{B}$ is defined as  
\[
\mathcal{B}(u) := \frac{1}{2} \partial_x(u^2).  
\]
we recall an existence result for the problem \eqref{diff_form_ctrl_prblm}.
\begin{proposition}\label{prp_existence}
	Let $(j, s) \in \{1, 2\} \times \mathbb{N}$, and suppose $T > 0$, \( u^0 \in H^s(\mathbb{T}) \), and \( f \in L^2((0, T); H^s_0(\mathbb{T})) \) are given. Then the system \eqref{diff_form_ctrl_prblm} admits a unique solution in \( C([0, T]; H^s(\mathbb{T})) \).  
\end{proposition}
For each \( j \), let \( \mathcal{R}_j \) be the operator that maps the initial condition \( (u^0, f) \) to the corresponding solution \( u \) of \eqref{diff_form_ctrl_prblm}. The solution evaluated at time \( t \) is represented by \( \mathcal{R}_{j, t}(u^0, f) \).
\begin{proposition}\label{prp_stability}  
	Let \( s \in \mathbb{N} \), \( T > 0 \), and \( f \in L^2((0, T); H^s_0(\mathbb{T})) \). For any initial data \( u^0, v^0 \in H^s(\mathbb{T}) \) satisfying \( [u^0] = [v^0] \), there exists a constant \( C(\|u^0\|_s, \|v^0\|_s) > 0 \) such that  
	\begin{align}\label{stability}  
		\sup_{t \in [0, T]} \|\mathcal{R}_{j, t}(u^0, f) - \mathcal{R}_{j, t}(v^0, f)\|_s \leq C \Big(\|u^0 - v^0\|_s \Big), \quad \forall j \in \{1, 2\}.  
	\end{align}  
\end{proposition}
The proofs of Proposition \ref{prp_existence} and Proposition \ref{prp_stability} for \( j = 1,2 \) were carried out in \cite{Chen23,Ahamed_Mondal_25}, respectively.

Building on the strategy developed in \cite{Nersesyan_2021_Sicon}, we aim to establish the approximate controllability of KdV-type equations of the form 
\begin{align}\label{ctrl_prblm_without_initial_fr_proof}
	\partial_t u + \mathcal{L}_j u + \mathcal{B}(u) = \eta
\end{align}  
by reducing the problem to the controllability of a linearized system derived from the inviscid Burgers equation. As a starting point, we examine the inviscid Burgers equation given by
\begin{align}\label{burger_equn}
	\partial_t w + \mathcal{B}(w) = \xi,
\end{align}  
and analyze its linearization along a solution trajectory
\begin{align}\label{linearized_Burger}
	\partial_t v + \mathcal{Q}(v, w) = g,
\end{align}  
which defines a time-dependent linear system. Here, the bilinear operator \(\mathcal{Q}(v, w)\) is defined by 
\begin{align}\label{defn_Q_B(v,w)}
	\mathcal{Q}(v, w) : = w \partial_x v + v \partial_x w.
\end{align}  
The control functions \(\eta, \xi, g\) are assumed to take values in a common finite-dimensional subspace of $C^{\infty}(\mathbb{T})$.

\section{Transport equation with finite-dimensional control }\label{sec_3}
We shall obtain in this section a controllability result for linear transport equations of the form 
\begin{align*}
	\partial_t v + \mathcal{Q}(v, w) = g,
\end{align*} 
where $g$ is a finite-dimensional control and $w$ denotes a suitably constructed return trajectory, as defined below in Section 3.2.
\subsection{Saturation} Given a subspace
\begin{align*} 
	\mathcal{V} = \text{span} \{ \sin (x), \cos (x)\}, 
\end{align*} 
we define a nested chain of subspaces as follows:
\begin{align}
	\mathcal{V}_0 &:= \mathcal{V}\\
	\mathcal{V}_k &:= \operatorname{span}_{\mathbb{R}} \left\{ \psi_1 + \mathcal{Q}(\psi_2, \phi) \ :\ \psi_1, \psi_2 \in \mathcal{V}_{k-1},\ \phi \in \mathcal{V}_0 \right\}, \ k \geq 1, \text{ and } \mathcal{V}_{\infty} : = \bigcup_{k=0}^\infty \mathcal{V}_k \label{H_k_finite_subspace}
\end{align}
where \( \mathcal{Q} \) is the bilinear form defined in \eqref{defn_Q_B(v,w)}. From this construction, it is clear that  
\begin{align}\label{inclusion_finite_space}  
	\mathcal{V}_0 \subset \mathcal{V}_1 \subset \cdots \subset \mathcal{V}_n \subset \cdots \subset \mathcal{V}_\infty.
\end{align}

\begin{definition}\label{saturating}
	A space \(\mathcal{V}\) is said to be saturating if \(\mathcal{V}_{\infty}\) is dense in \(H^s(\mathbb{T})\) for all \(s \geq 0\).
\end{definition}

\begin{proposition}[Density]\label{prp_saturating}
The spaces \(\mathcal{V}_n\) satisfy the following properties:
\begin{itemize}
\item[(i)] For \(n \geq 1\), \(\mathcal{V}_n\) is a vector space of dimension \(2(n + 1) \). Specifically, 
\begin{align}\label{space_V_n}
\mathcal{V}_n = \mathrm{span}\{ \sin(x), \cos(x), \dots, \sin((n + 1) x), \cos((n + 1) x)\}\footnotemark.
\end{align}
\item[(ii)] The space \(\mathcal{V}\) is saturating, that is, \(\mathcal{V}_{\infty}\) is dense in \(H^s(\mathbb{T})\).
\end{itemize}
\end{proposition}
\footnotetext{Here, just note that $\mathcal{V}_1 = \mathcal{H}$, as defined \eqref{defn_subspace_H}.}
\begin{proof}
We prove the result by induction. Before proceeding with the induction step, we first derive an identity using trigonometric formulas. For $a,b,c,d \in \mathbb{R}$ and $k,m \in \mathbb{Z}$, we have
\begin{align}
&\mathcal{Q}\big(a \sin(mx) + b \cos(kx),\, c \sin(x) + d \cos(x)\big) = \frac{a}{2} (m + 1) \Big[ c \sin((m + 1)x) + d \cos((m + 1)x)\Big]  \notag \\
& \qquad + \frac{b}{2} (k + 1) \Big[ c \cos((k + 1)x) - d \sin((k + 1)x)\Big]  + \frac{a}{2} (m - 1) \Big[ d \cos((m - 1)x) - c \sin((m - 1)x)\Big] \notag \\
&\qquad \qquad \qquad+ \frac{b}{2} (1 - k) \Big[ c \cos((k - 1)x) - d \sin((k - 1)x)\Big].
\end{align}
Using this identity, the desired result follows by induction. We omit the details, as the argument proceeds in the same spirit as in \cite{Chowdhury_Dutta_Mondal_25}.
\end{proof}

\subsection{Return trajectories obtained from observable families} We recall the notion of an \emph{observable family}, introduced in \cite{Kuksin_Nersesyan_Shirikyan_2020} in the study of ergodic properties of randomly forced partial differential equations. This concept was later applied in \cite{Nersesyan_2021_Sicon} to the controllability of the Navier--Stokes equations.
\begin{definition}\label{defn_Observable_family}
A finite collection \( \{ \varphi_i \}_{i=1}^n \subset L^2([0, T]; \mathbb{R}) \) is said to be an \textit{observable family} if the following property holds: for every subinterval \( J \subset [0, T] \), and for any function \( b \in C^0(J; \mathbb{R}) \) together with functions \( a_i \in C^1(J; \mathbb{R}) \), the condition
\begin{align}\label{observable_cndn}
b(t) + \sum_{i=1}^n a_i(t) \varphi_i(t) = 0 \quad \text{in } L^2(J; \mathbb{R})
\end{align}
implies that \( b \equiv 0 \) and \( a_i \equiv 0 \) for all \( 1 \le i \le n \) on \( J \).
\end{definition}
\begin{remark}
As explained in \cite[Section 3.3]{Nersesyan_2021_Sicon}, see also \cite{Kuksin_Nersesyan_Shirikyan_2020}, one can construct observable families in the sense of Definition~\ref{defn_Observable_family}. For another explicit construction of such functions, see \cite[Remark 4.31]{Rudin_analysis_book}.
\end{remark}

\textbf{Construction of return trajectory}. Here we prove the following proposition
\begin{proposition}
	Let $T>0$ and let $\mathcal{H}$ be the subspace defined in \eqref{defn_subspace_H}. Then there exist a control function $\xi \in L^2((0,T);\mathcal{H})$ and a corresponding solution 
	\[
	w \in W^{1,2}((0,T);C^{\infty}(\mathbb{T}))\footnotemark
	\]
	to equation \eqref{burger_equn} such that
	\begin{align}
		& w(0)=0 \quad \text{and} \quad w(T)=0, \label{w_zero_both_end}\\
		& \mathcal{L}_j w(t) \in \mathcal{H}, \qquad t\in[0,T], \; j\in\{1,2\}. \label{L_j_w_in_H}
	\end{align}
\end{proposition}
\footnotetext{For any Banach space $X$, $W^{1,2}((0,T);X)$ denotes the space of functions $u:(0,T)\to X$ such that $\frac{d^i}{dt^i}u \in L^2((0,T);X)$ for $i=0,1$. The space $W^{1,2}_0((0,T);X)$ consists of functions in $W^{1,2}((0,T);X)$ satisfying $u(0)=u(T)=0$.}
\begin{proof}
	For $T>0$, we fix an observable family $\{\varphi_s,\varphi_c\}\subset L^2((0,T);\mathbb{R})$ and choose 
	$\Theta:[0,T]\to\mathbb{R}$ be a $C^1$ function such that $\Theta(t)=0$ if and only if $t=T$. Define
	\[
	\vartheta_s(t)=\Theta(t)\int_0^t \varphi_s(\rho)\,d\rho,
	\qquad
	\vartheta_c(t)=\Theta(t)\int_0^t \varphi_c(\rho)\,d\rho,
	\qquad t\in[0,T].
	\]
	Next, define
	\begin{align}\label{expression_w}
		w(t,x):=\vartheta_s(t)\sin x+\vartheta_c(t)\cos x .
	\end{align}
	Set
	\[
	\xi:=\partial_t w+\mathcal{B}(w).
	\]
	We claim that the pair $(w,\xi)$ satisfies the desired properties.
	
	First, since
	\[
	\vartheta_s(0)=\vartheta_s(T)=\vartheta_c(0)=\vartheta_c(T)=0,
	\]
	it follows immediately that
	\[
	w(0)=w(T)=0,
	\]
	so condition \eqref{w_zero_both_end} holds.
	
	Moreover, since $\vartheta_s,\vartheta_c\in W^{1,2}_0((0,T); \mathbb{R})$, we have
	\[
	w\in W^{1,2}_0((0,T);C^\infty(\mathbb{T})).
	\]
	
	Next, recall that the space $\mathcal{V}_0$ is invariant under the operators $\mathcal{L}_j$ for all $j\in\{1,2\}$. Since $w(t,\cdot)\in\mathcal{V}_0$ for  $t\in[0,T]$, we obtain
	\[
	\mathcal{L}_j w(t)\in\mathcal{V}_0\subset\mathcal{V}_1=\mathcal{H}
	\qquad
	\text{for } t\in[0,T],\; j\in\{1,2\}.
	\]
	Hence condition \eqref{L_j_w_in_H} is satisfied.
	
	It remains to verify that $\xi\in L^2((0,T);\mathcal{H})$. By construction, we have
	\[
	\partial_t w\in L^2((0,T);\mathcal{H}).
	\]
	We now examine the nonlinear term $\mathcal{B}(w)$. Since
	\[
	w(t,x)=\vartheta_s(t)\sin x+\vartheta_c(t)\cos x,
	\]
	we compute
	\begin{align*}
		\mathcal{B}(w)
		&=w\,w_x\\
		&=\left(\vartheta_s(t)\sin x+\vartheta_c(t)\cos x\right)
		\partial_x\left(\vartheta_s(t)\sin x+\vartheta_c(t)\cos x\right)\\
		&=\frac{\vartheta_s^2(t)-\vartheta_c^2(t)}{2}\sin(2x)
		+\vartheta_s(t)\vartheta_c(t)\cos(2x).
	\end{align*}
	Therefore,
	\[
	\mathcal{B}(w)\in C([0,T];\mathcal{H}).
	\]
	Consequently,
	\[
	\xi=\partial_t w+\mathcal{B}(w)\in L^2((0,T);\mathcal{H}),
	\]
	which completes the proof.
\end{proof}

\begin{remark}
	We note that, in contrast with \cite{Chen23,Ahamed_Mondal_25}, where the global approximate controllability of the KdV equation ($j=1$) and the Kawahara equation ($j=2$) is achieved using a two-dimensional control space, the construction of the return trajectory $w$ in \eqref{expression_w} requires a four-dimensional space.
\end{remark}

We next show that $\mathcal{H}$-valued controls can generate all desired frequencies when acting on linear transport equations, which is achieved by linearizing the inviscid Burgers equation \eqref{burger_equn} along the solution trajectory $w$ constructed above in \eqref{expression_w}.

\subsection{Existence of finite-dimensional control} Given $T > 0, \ w \in W_0^{1,2}((0,T);C^\infty(\mathbb{T}))$ consider the equation
\begin{align}\label{equ_30}
\begin{cases}
	\partial_t v + \mathcal{Q}(v, w) = g,\\
	v(0) = 0.
\end{cases}
\end{align}
\begin{proposition}\label{prpn_approx_linear_equn}
The equation \eqref{equ_30} is approximately controllable in time \( T > 0 \) using controls taking values in \( \mathcal{H} \). More precisely, for every target state \( v^1 \in H^{s}_0 \) and \( \varepsilon > 0 \), there exists a control input \( g \in L^2((0, T); \mathcal{H}) \) such that the solution $v \in C([0, T]; H^{s}_0)$ of  \eqref{equ_30}, satisfies 
	\[
	\|v(T) - v^1\|_{H^s} \leq \varepsilon.
	\]
\end{proposition}
\begin{proof}
This approach is inspired by the work of Nersesyan \cite{Nersesyan_2021_Sicon}, where approximate controllability of the 3D Navier–Stokes system is established. We denote by  
\[
\mathbf{R}(t, \delta) : H^s_0 \to H^s_0, \qquad v_0 \mapsto v(t),  \quad 0 \leq \delta \leq t \leq T,
\]  
the associated two-parameter resolving operator for the problem  
\begin{align}\label{linearized_burger_start_at_delta}
	\dot{v} + \mathcal{Q}(w, v) = 0, \quad v(\delta) = v^0.
\end{align}  
Consequently, the resolving operator 
\[
\mathcal{A} : L^2((0, T); H^s_0(\mathbb{T})) \to H^s_0(\mathbb{T}), 
\]
 providing the solution to \eqref{equ_30} at time $t = T$ by the Duhamel's principle 
\begin{align*}
\mathcal{A} g = \int_0^T \mathbf{R}(T, \delta)\, g(\delta) \, d\delta, \quad g \in L^2((0, T); H^s_0).
\end{align*}
Let \( \mathcal{P}_{\mathcal{H}} : H^s_0(\mathbb{T}) \to \mathcal{H} \subset H^s_0(\mathbb{T}) \) denote the orthogonal projection onto the subspace \( \mathcal{H} \) with respect to $H^s$. We then define the operator  
\[
\mathcal{A}_1 := \mathcal{A} \mathcal{P}_{\mathcal{H}} : L^2((0, T); H^s_0(\mathbb{T})) \to H^s_0(\mathbb{T}),
\]
and aim to show that \( \operatorname{Im}(\mathcal{A}_1) \) is dense in \( H^s_0(\mathbb{T}) \). It is equivalent to prove that the kernel of the adjoint operator \( \mathcal{A}_1^* \) is trivial. The adjoint \( \mathcal{A}_1^* \) is given by  
\[
\mathcal{A}_1^* : H^s_0(\mathbb{T}) \to L^2((0, T); H^s_0(\mathbb{T})), \quad (\mathcal{A}_1^*z) (\cdot) = \mathcal{P}_{\mathcal{H}} \mathbf{R}(T, \cdot)^* z,
\]
where  \( \mathbf{R}(T, \delta)^* \) denotes for $\delta \in [0, T]$ the adjoint of \( \mathbf{R}(T, \delta) \) with respect to the \( H^s \)-inner product.

\underline{Triviality of $\operatorname{Ker} \mathcal{A}_1^*$ :} Let us consider an element \( z \in \ker \mathcal{A}_1^* \). Our objective is to demonstrate that \( z = 0 \). 

By definition of the kernel, for any \( g \in \mathcal{H} \), we have
\[
\left\langle g, \mathbf{R}(T, \delta)^* z \right\rangle_s = 0, \quad \text{for almost every } \delta \in [0, T].
\]
Due to the continuity in \( \delta \) of the map \( \delta \mapsto \mathbf{R}(T, \delta)g \) in the topology of \( H^s \), we conclude that the pairing vanishes identically
\begin{equation} \label{eq:innerproduct_zero_all_delta}
	\left\langle \mathbf{R}(T, \delta)g, z \right\rangle_s = 0, \quad \text{for all } \delta \in [0, T].
\end{equation}
Fix some time \( T_1 \in (0, T) \). Using the semigroup (or flow) property of the resolvent operator, we may write
\[
\mathbf{R}(T, \delta) = \mathbf{R}(T, T_1)\mathbf{R}(T_1, \delta), \quad \text{for } \delta \in [0, T_1].
\]
Substituting into \eqref{eq:innerproduct_zero_all_delta}, we obtain
\[
\left\langle \mathbf{R}(T, T_1)\mathbf{R}(T_1, \delta)g, z \right\rangle_s = 0.
\]
Rewriting this using adjoint properties, define \( z_1 := \mathbf{R}(T, T_1)^* z \), so that
\begin{equation} \label{eq:T1_delta_innerproduct}
	\left\langle \mathbf{R}(T_1, \delta)g, z_1 \right\rangle_s = 0, \quad \text{for all } \delta \in [0, T_1].
\end{equation}
In particular, taking \( \delta = T_1 \), we find
\[
\left\langle g, z_1 \right\rangle_s = 0,
\]
which implies $z_1 \perp \mathcal{H} \quad \text{in } H^s,$
by \eqref{space_V_n}, we have
\begin{equation} \label{eq:z1_orthogonal_H1}
	z_1 \perp \mathcal{V}_1 \quad \text{in } H^s,
\end{equation}

Our next objective is to establish that \( z_1 \) is orthogonal to the space \( \mathcal{V}_2 \). To proceed, we introduce the following notation
\begin{equation} \label{eq:def_y}
	y(t, \delta) := \mathbf{R}(\delta + t, \delta)g,
\end{equation}
where \( y(t, \delta) \) denotes the solution at time \( \delta + t \) of the linearized evolution equation, starting from initial data \( g \) at time \( \delta \). Then \( y \) satisfies the initial value problem
\begin{align}
	\partial_t y(t, \delta) + \mathcal{Q}(w(\delta + t), y(t, \delta)) &= 0, \quad t \in (0, T - \delta), \label{eq:y_evolution}\\
	y(0, \delta) &= g. \label{eq:y_initial}
\end{align}
Define:
\[
Y(t, \delta) := \partial_\delta y(t, \delta),
\]
which satisfies the following linear PDE, obtained by differentiating \eqref{eq:y_evolution}
\begin{align}
	\partial_t Y(t, \delta) + \mathcal{Q}(w(\delta + t), Y(t, \delta)) + \mathcal{Q}(\dot{w}(\delta + t), y(t, \delta)) &= 0, \quad t \in (0, T - \delta), \label{eq:Y_evolution}\\
	Y(0, \delta) &= 0. \label{eq:Y_initial}
\end{align}
On the other hand, differentiating \eqref{eq:def_y} with respect to \( \delta \), and evaluating at \( t = T_1 - \delta \), we compute
\begin{align}
	\partial_\delta \mathbf{R}(T_1, \delta)g &= \frac{d}{d\delta} \left[\mathbf{R}(\delta + t, \delta)g\right] \bigg|_{t = T_1 - \delta} \notag\\
	&= \partial_1 \mathbf{R}(T_1, \delta)g + \partial_2 \mathbf{R}(T_1, \delta)g \notag\\
	&= -\mathcal{Q}(w(T_1), \mathbf{R}(T_1, \delta)g) + Y(T_1 - \delta, \delta), \label{derivative_R_delta}
\end{align}
where the first term we have used the identity from the linearized operator \eqref{linearized_burger_start_at_delta}
\[
\partial_t \mathbf{R}(T_1, \delta)g = -\mathcal{Q}(w(T_1), \mathbf{R}(T_1, \delta)g),
\]
and the second from the definition of \( Y \). 
Differentiating the identity \eqref{eq:T1_delta_innerproduct} with respect to \( \delta \), we obtain:
\begin{equation}\label{eq:delta_inner_product_derivative}
	\frac{d}{d\delta} \left\langle \mathbf{R}(T_1, \delta)g, z_1 \right\rangle_s = 0.
\end{equation}
Since \( z_1 \) is independent of \( \delta \), applying the product rule yields:
\[
\left\langle \frac{d}{d\delta} \mathbf{R}(T_1, \delta)g, z_1 \right\rangle_s = 0.
\]
Using the identity \eqref{derivative_R_delta}, we substitute the expression for the derivative of \( \mathbf{R}(T_1, \delta)g \)
\[
\frac{d}{d\delta} \mathbf{R}(T_1, \delta)g = Y(T_1 - \delta, \delta) - \mathcal{Q}(w(T_1), \mathbf{R}(T_1, \delta)g),
\]
where \( Y \) satisfies the linear equation \eqref{eq:Y_evolution}–\eqref{eq:Y_initial}. Substituting into \eqref{eq:delta_inner_product_derivative}, we obtain
\begin{equation}\label{eq:derivative_identity}
	\left\langle Y(T_1 - \delta, \delta), z_1 \right\rangle_s = \left\langle \mathcal{Q}(w(T_1), \mathbf{R}(T_1, \delta)g), z_1 \right\rangle_s.
\end{equation}

To express \( Y(T_1 - \delta, \delta) \) in terms of its time evolution, we integrate equation \eqref{eq:Y_evolution} over the interval \( [0, T_1 - \delta] \). Using the initial condition \( Y(0, \delta) = 0 \), we obtain
\begin{align}
	Y(T_1 - \delta, \delta) &= \int_0^{T_1 - \delta} \partial_t Y(t, \delta) \, dt \notag\\
	&= - \int_0^{T_1 - \delta} \Big( \mathcal{Q}(w(\delta + t), Y(t, \delta)) + \mathcal{Q}(\dot{w}(\delta + t), y(t, \delta)) \Big) \, dt. \label{eq:Y_integrated}
\end{align}
Substituting \eqref{eq:Y_integrated} into \eqref{eq:derivative_identity}, and recalling from \eqref{eq:def_y} that \( y(t, \delta) = \mathbf{R}(\delta + t, \delta)g \), we obtain the following integral identity
\begin{align}
	\int_0^{T_1 - \delta} \left\langle \mathcal{Q}(w(\delta + t), Y(t, \delta)), z_1 \right\rangle_s \, dt  + \int_0^{T_1 - \delta} \left\langle \mathcal{Q}(\dot{w}(\delta + t), \mathbf{R}(\delta + t, \delta)g), z_1 \right\rangle_s \, dt + \left\langle \mathcal{Q}(w(T_1), \mathbf{R}(T_1, \delta)g), z_1 \right\rangle_s = 0. \label{eq:final_integral_identity}
\end{align}
Changing variables in the second integral via \( t' = \delta + t \), we rewrite it as
\[
\int_\delta^{T_1} \left\langle \mathcal{Q}(\dot{w}(t'), \mathbf{R}(t', \delta)g), z_1 \right\rangle_s \, dt'.
\]
So equation \eqref{eq:final_integral_identity} becomes
\begin{align}
	\int_0^{T_1 - \delta} \left\langle \mathcal{Q}(w(\delta + t), Y(t, \delta)), z_1 \right\rangle_s \, dt + \int_\delta^{T_1} \left\langle \mathcal{Q}(\dot{w}(t), \mathbf{R}(t, \delta)g), z_1 \right\rangle_s \, dt
	+ \left\langle \mathcal{Q}(w(T_1), \mathbf{R}(T_1, \delta)g), z_1 \right\rangle_s = 0. \label{eq:integral_identity_explicit}
\end{align}
Differentiating the integral identity \eqref{eq:integral_identity_explicit} with respect to \( \delta \), we obtain the following equality
\begin{equation}\label{eq:delta_differentiated_form}
	b(\delta) + a_s(\delta) \varphi_s(\delta) + a_c(\delta) \varphi_c(\delta)  = 0, \quad \text{for all } \delta \in [0, T_1],
\end{equation}
where
\begin{align*}
	b(\delta) &= \frac{d}{d\delta} \int_{0}^{T_1 - \delta} \left\langle \mathcal{Q}(w(\delta + t), Y(t, \delta)), z_1 \right\rangle_s \, dt  - \left\langle \mathcal{Q}( \dot{\varTheta}(\delta)   \int_0^{\delta} \left( \varphi_s(t) \sin(x) + \varphi_c(t) \cos(x) \right) \, dt, g), z_1 \right\rangle_s \\
	&\quad \hspace{4cm} + \int_{\delta}^{T_1 - \delta} \left\langle \mathcal{Q}(\dot{w}(t), \partial_\delta \mathbf{R}(t, \delta)g), z_1 \right\rangle_s \, dt + \frac{d}{d\delta} \left\langle \mathcal{Q}(w(T_1), \mathbf{R}(T_1, \delta)g), z_1 \right\rangle_s, \\
	a_s(\delta) &= -\varTheta(\delta) \left\langle \mathcal{Q}(\sin(x), g), z_1 \right\rangle_s, \\
	a_c(\delta) &= - \varTheta(\delta) \left\langle \mathcal{Q}(\cos(x), g), z_1 \right\rangle_s.
\end{align*}
Here, \( b \) is a continuous function on \( [0, T_1] \), and each \( a_{s,c} \in C^1([0, T_1]) \).

Now, by the observability property of the functions \( \{ \varphi_s, \varphi_c \} \), the identity \eqref{eq:delta_differentiated_form} implies that
\[
a_s \equiv 0 \equiv a_c \quad \text{on } [0, T_1], 
\]
since \( \varTheta(\delta) \neq 0 \) on \( [0, T_1] \), which is assumed. As a result, we conclude
\[
\left\langle \mathcal{Q}(\sin(x), g), z_1 \right\rangle_s = \left\langle \mathcal{Q}(\cos(x), g), z_1 \right\rangle_s = 0,
\]
for any \( g \in \mathcal{V}_1 \).
Combined with the orthogonality condition \eqref{eq:z1_orthogonal_H1}, and the definition of $\mathcal{V}_2$ we deduce that \( z_1 \) is orthogonal to the entire subspace $\mathcal{V}_2.$
Recalling the definition \( z_1 = \mathbf{R}(T, T_1)^* z \), we may rewrite the resulting orthogonality in terms of the adjoint map
\[
\left\langle \mathbf{R}(T, T_1)g, z \right\rangle_s = 0, \quad \text{for all } g \in \mathcal{V}_2, \text{ and } T_1 \in (0, T).
\]
Renaming \( T_1 \) as \( \delta \), this yields
\[
\left\langle \mathbf{R}(T, \delta)g, z \right\rangle_s = 0, \quad \forall g \in \mathcal{V}_2,\ \delta \in [0, T],
\]
which extends the orthogonality condition \eqref{eq:innerproduct_zero_all_delta} to \( \mathcal{V}_2 \).
By iterating this argument inductively for \( \mathcal{V}_k \) with increasing \( k \), we extend the conclusion
\[
\left\langle \mathbf{R}(T, \delta)g, z \right\rangle_s = 0, \quad \forall g \in \bigcup_{k=0}^\infty \mathcal{V}_k,\ \delta \in [0, T].
\]

Finally, taking \( \delta = T \) and applying the density of the union \( \bigcup_k \mathcal{V}_k \) we have,
\[
\left\langle g, z \right\rangle_s = 0, \quad \forall g \in H^s \quad \Rightarrow \quad z = 0.
\]
This completes the proof. 
\end{proof}

\section{Closeness of solutions of equations \eqref{ctrl_prblm_without_initial_fr_proof} and \eqref{linearized_Burger} in small time via large controls}\label{sec_4}
To deduce the controllability of the nonlinear KdV-type equation from the approximate controllability of the linearized system by finite-dimensional controls, we establish the following proposition.
\begin{proposition}\label{prpn_closeness_solun}
	Let \( j \in \{1, 2\}, \, s \in \mathbb{N} \), and \( \mathcal{H} \) be as \eqref{defn_subspace_H}. Then, for any initial state \( u^0 \in H^{s + 2j + 1}_0(\mathbb{T}) \), any control function \( g \in L^2((0, T); \mathcal{H}) \),  there exist sufficiently small \( \tau > 0 \) and a control \( \eta_{\tau, j} \in L^2((0, T\tau); \mathcal{H}) \) such that  
	\begin{align}\label{lim_closness_solun}
		\mathcal{R}_{j, T\tau}(u^0, \eta_{\tau, j})  \to v(T) \quad \text{in } H^s(\mathbb{T}) \quad \text{as } \tau \to 0^+,
	\end{align}
	
	where \( v \in C([0, T]; H^{s + 2j + 1}_0) \) solves the linearized equation \eqref{linearized_Burger} with initial condition \( v(0) = u^0 \), see Figure~\ref{closeness_fig}.  
	\begin{figure}[h!]
		\centering
		\begin{tikzpicture}[x=0.8 cm, y=0.8cm]
			\begin{pgfonlayer}{nodelayer}
				\node [style=none] (0) at (0, 6.5) {};
				\node [style=none] (1) at (0, -1) {};
				\node [style=none] (2) at (-1, 0) {};
				\node [style=none] (3) at (8.5, 0) {};
				\node [style=none] (4) at (7, 0) {};
				\node [style=none] (5) at (3.25, 0) {};
				\node [style=Dotty,scale=0.45] (6) at (7, 4) {};
				\node [style=Dotty,scale=0.45] (7) at (0, 0) {};
				\node [style=none] (8) at (0, 4) {};
				\node [style=none] (9) at (0, 5) {};
				\node [style=none] (10) at (0, 3) {};
				\node [style=none] (11) at (7, 3) {};
				\node [style=none] (12) at (7, 5) {};
				\node [style=Dotty,scale=0.45] (13) at (3.25, 4.5) {};
				\node [style=none] (14) at (4.25, 1.75) {};
				\node [style=none] (15) at (1.75, 2.25) {};
				\node [style=none] (16) at (-1.25, 5) {};
				\node [style=none] (17) at (-1.25, 3) {};
				\node [style=none] (18) at (-1, 3) {};
				\node [style=none] (19) at (-1, 5) {};
				\node [style=none] (20) at (-2.75, 4) {$\varepsilon$-nbd of $v(T)$};
				\node [style=none] (21) at (-0.65, 4) {$v(T)$};
				\node [style=none] (22) at (-0.5, 0.5) {$u^0$};
				\node [style=none] (23) at (3.25, -0.5) {$T\tau$};
				\node [style=none] (24) at (7, -0.5) {$T$};
				\node [style=none] (25) at (5, 1.5) {$g$};
				\node [style=none] (26) at (1.6, 1.68) {$\eta_{\tau, j}$ };
				\node [style=none] (27) at (7.75, 4) {$v(T)$};
				\node [style=none] (28) at (4.7, 4.5) {$\ \mathcal{R}_{j, T\tau}(u^0, \eta_{\tau, j})$};
				\node [style=none] (29) at (-8.75, 3.25) {};
				\node [style=none] (30) at (-8, 3.25) {};
				\node [style=none] (31) at (-8.75, 2.25) {};
				\node [style=none] (32) at (-8, 2.25) {};
				\node [style=none] (33) at (-5.75, 3.25) {Nonlinear trajectory};
				\node [style=none] (34) at (-5.95, 2.25) {Linear trajectory};
			\end{pgfonlayer}
			\begin{pgfonlayer}{edgelayer}
				\draw (0.center) to (1.center);
				\draw (2.center) to (3.center);
				\draw [dashed](9.center) to (12.center);
				\draw [dashed](10.center) to (11.center);
				\draw [dashed](8.center) to (6.center);
				\draw [dashed](13.center) to (5.center);
				\draw [dashed](6.center) to (4.center);
				\draw [style=eddt, in=-180, out=15, looseness=1.25,red] (7.center) to (14.center);
				\draw [in=225, out=0,red] (14.center) to (6.center);
				\draw [style=eddt, in=-150, out=45, looseness=1.50,blue] (7.center) to (15.center);
				\draw [in=-165, out=30, looseness=1.25,blue] (15.center) to (13.center);
				\draw (16.center) to (19.center);
				\draw (17.center) to (18.center);
				\draw (16.center) to (17.center);
				\draw [in=-120, out=30, looseness=1.25,blue] (29.center) to (30.center);
				\draw [in=-120, out=30, looseness=1.25,red] (31.center) to (32.center);
			\end{pgfonlayer}
		\end{tikzpicture}
		
		\caption{Closeness of solutions to \eqref{ctrl_prblm_without_initial_fr_proof} and \eqref{linearized_Burger}}
		\label{closeness_fig}
	\end{figure}
	
	Furthermore, the control \( \eta_\tau \) is explicitly given by  
	\begin{align}\label{form_ctrl_as_solun_linear_equn}
		\eta_{\tau, j}(t) = \tau^{-1} g(\tau^{-1} t) + \tau^{-2} \xi(\tau^{-1} t) + \mathcal{L}_j w(\tau^{-1} t), \quad t \in [0, T\tau].
	\end{align}
	
\end{proposition}

\begin{proof}[Proof of Proposition \ref{prpn_closeness_solun}] 
	Let \( M > 0 \), \( u^0 \in B_{H^{s + 2j + 1}_0(\mathbb{T}) }(0, M)\), and \( \eta_j \in L^2_{\mathrm{loc}}(\mathbb{R}^+, \mathcal{H}) \) be arbitrary. Recall that
	\[
	u_j(t) = \mathcal{R}_{j, t}(u^0, \eta_j), \quad t \in [0, T],
	\]
	the solution to \eqref{ctrl_prblm_without_initial_fr_proof} with initial condition \( u(0) = u^0 \).
	
	Following the rescaling technique from \cite{Nersesyan_2021_Sicon}, for simplicity we fix \( \tau \in (0, 1) \) and define the time-rescaled functions
	\begin{align}\label{tau_time_scaling}
		v_\tau(t) := v(\tau^{-1}t), \qquad &g_\tau(t) := \tau^{-1} g(\tau^{-1}t), \notag \\
		w_\tau(t) := \tau^{-1} w(\tau^{-1}t), \qquad &\xi_\tau(t) := \tau^{-2} \xi(\tau^{-1}t), \notag \\
		r_j(t) := u_j(t) - v_\tau(t) - w_\tau(t), \qquad &t \in [0, T\tau].
	\end{align}
By construction and in view of \eqref{w_zero_both_end} and the identity \( v(0) = u^0 \), it follows that
\begin{equation}\label{boundary_of_r}
	r_j(0) = 0, \qquad r_j(T\tau) = u_j(T\tau) - v(T).
\end{equation}
Therefore, in order to establish the limit \eqref{lim_closness_solun}, it suffices to construct, for sufficiently small \( \tau > 0 \), a control \( \eta_{\tau, j} \in L^2((0, T\tau); \mathcal{H}) \) such that:
\begin{itemize}
	\item[(i)] The map \( \mathcal{R}_{j, t}(u^0, \eta_{\tau,j}) \) is well-defined at $t = T\tau$;
	\item[(ii)] The corresponding solution satisfies
	\begin{equation}\label{limit_of_r}
		\|r_j(T\tau)\|_s \longrightarrow 0 \quad \text{as } \tau \to 0^+,
	\end{equation}
	uniformly with respect to all initial data \( u^0 \in  B_{H^{s + 2j + 1}_0(\mathbb{T}) }(0, M) \), for each fixed \( s \in \mathbb{N} \).
\end{itemize}

The functions \( v_{\tau}(t) \) and \( w_{\tau}(t) \), defined on the interval \( [0, T] \), satisfy the following equations:
\begin{align*}
	\partial_t v_{\tau} + \mathcal{Q}(v_{\tau}, w_{\tau}) &= g_{\tau}, \\
	\partial_t w_{\tau} + \mathcal{B}(w_{\tau}) &= \xi_{\tau}.
\end{align*}

It follows that the remainder \( r_j(t) := u_j(t) - v_{\tau}(t) - w_{\tau}(t) \) solves the equation
\begin{equation}\label{equn_r}
	\partial_t r_j + \mathcal{L}_j r_j + \mathcal{B}(r_j) + \mathcal{Q}(r_j, v_{\tau} + w_{\tau}) = \zeta_{\tau}, \quad t \in [0, T\tau],
\end{equation}
where the term \( \zeta_{\tau} \) is given by
\begin{align*}
	\zeta_{\tau} = \eta_{\tau, j} - \mathcal{L}_j(v_{\tau} + w_{\tau}) - \mathcal{B}(v_{\tau}) - g_{\tau} - \xi_{\tau}.
\end{align*}

We now choose the control \( \eta_{\tau, j} \in L^2((0, T\tau); \mathcal{H}) \) as
\begin{equation}\label{choice_of_control}
	\eta_{\tau, j} := g_{\tau} + \xi_{\tau} + \mathcal{L}_j(w_{\tau}), \qquad 
\end{equation}
in accordance with \eqref{form_ctrl_as_solun_linear_equn}. Substituting this expression into the formula for \( \zeta_{\tau} \), we obtain
\begin{equation}\label{source_term_zeta}
	\zeta_{\tau} = -\mathcal{L}_j(v_{\tau}) - \mathcal{B}(v_{\tau}).
\end{equation}
Since $g$, $\xi$, and $w$ are preassigned and defined on $[0,T]$, it follows that $\eta_{\tau,j} \in L^2((0,T\tau);\mathcal{H})$. Therefore, by Proposition~\ref{prp_existence}, 
the map $\mathcal{R}_{j,t}(u_0,\eta_{\tau,j})$ is well defined at $t = T\tau$. 
From now on, we use the notation $P \lesssim Q$ to mean that there exists a constant $C>0$ such that $P \le C Q$.

\underline{Case \( j \in \{1, 2\}, \ s = 0 \)}. We multiply \eqref{equn_r} by \( r_j \), integrate over \( \mathbb{T} \), and then integrate in time over the interval \( t \in (0, T\tau] \). Using \eqref{boundary_of_r}, we obtain  
\begin{align} \label{est_s = 0}
	\| r_j(t)\|^2 & \lesssim \int_{0}^{t} \|r_j\|^2 \| v_{\tau} + w_{\tau} \|_{2} \, d\rho + \int_{0}^{t} \|r_j\| \| v_{\tau} \|_{2j + 1} \, d\rho + \int_{0}^{t} \|v_{\tau}\|_1^2 \|r_j\| \, d\rho.
\end{align}
Applying Young’s inequality \( ab \leq \frac{a^p}{p} + \frac{b^q}{q} \), with  \( (p, q) = (4, \tfrac{4}{3}) \) for the last two, we infer
\begin{align*}
	\| r_j(t)\|^2 \lesssim \int_{0}^{t} \|r_j\|^2 \| v_{\tau} + w_{\tau} \|_{2} \, d\rho + \int_{0}^{t} \|v_{\tau}\|_{2j + 1}^{\frac{4}{3}} \, d\rho + \int_{0}^{t} \|v_{\tau}\|_{1}^{\frac{8}{3}} \, d\rho + \int_{0}^{t} \| r_j\|^4 \, d\rho.
\end{align*}
A change of variables using \eqref{tau_time_scaling} yields
\begin{align}\label{norm_v_tau}
	\int_{0}^{t} \|v_{\tau}\|_{2j + 1}^{\frac{4}{3}} \, d\rho \leq \tau \int_{0}^{T} \|v\|_{2j + 1}^{\frac{4}{3}} \, d\rho \qquad \text{ and } \qquad \int_{0}^{t} \|v_{\tau}\|_{1}^{\frac{8}{3}} \, d\rho \leq \tau \int_{0}^{T} \|v\|_{1}^{\frac{8}{3}} \, d\rho.
\end{align}
Since \( v \in C([0, T]; H_0^{2j + 1}) \) solves the linearized equation \eqref{linearized_Burger} with initial condition \( v(0) = u^0 \), where \( u^0 \in B_{H_0^{2j + 1}(\mathbb{T})}(0, M) \), it follows that there exists a constant \( \varepsilon_{\tau} = \varepsilon_{\tau}(M) > 0 \), independent of \( t \in [0, T\tau] \), such that \( \varepsilon_{\tau} \to 0 \) as \( \tau \to 0^+ \) and 
\begin{align*}
\int_{0}^{t} \|v_{\tau}\|_{2j + 1}^{\frac{4}{3}} \, d\rho + \int_{0}^{t} \|v_{\tau}\|_{1}^{\frac{8}{3}} \, d\rho \le \varepsilon_{\tau}.
\end{align*} 
Combining this with the previous estimates, we arrive at
\begin{align}\label{combining_s0_any_j}
	\| r_j(t) \|^2 \le \varepsilon_{\tau} + C \int_{0}^{t}  (\| v_{\tau} + w_{\tau} \|_{2}) \|r_j\|^2 \, d\rho + C \int_0^t \| r_j \|^4 \, d\rho, \quad \text{for all } t \in [0, T\tau].
\end{align}
By the Gronwall inequality,
\begin{align*} 
	\| r_j(t) \|^2 \le \left(\varepsilon_{\tau} + C \int_0^t \| r_j \|^4 \, d\rho\right) \exp \left(C \int_{0}^{t}  (\| v_{\tau} + w_{\tau} \|_{2})  \, d\rho \right).
\end{align*}
For $\tau \in (0,1)$ and $t \in [0,T\tau]$, using the facts that 
$v \in C([0,T]; H_0^{2j+1})$ and $w \in W^{1,2}_0((0,T); C^\infty)$, we obtain
\begin{align*}
	\int_{0}^{t} \| v_{\tau} + w_{\tau} \|_{2}\, d\rho
	&= \int_{0}^{t/\tau} \bigl(\tau \| v \|_{2} + \| w \|_{2}\bigr)\, d\rho \\
	&\le \int_{0}^{T} \bigl(\tau \| v \|_{2} + \| w \|_{2}\bigr)\, d\rho 
	\le C_0 ,
\end{align*}
where $C_0 = C_0(M) > 0$ does not depend on $t$, $\tau$, or $u_0$. 
By the same letter $\varepsilon_{\tau}$ we denote $e^{C_0}\varepsilon_{\tau}$. 
Thus
\begin{align}\label{fnl_est_s=0}
	\| r_j(t) \|^2 \le \varepsilon_{\tau} + C \int_0^t \| r_j(\rho) \|^4 \, d\rho, 
	\qquad t \in [0, T\tau].
\end{align}
Define the function
\[
\Psi(t) := \varepsilon_{\tau} + C \int_0^t \| r_j \|^4 \, d\rho.
\]
Then \eqref{fnl_est_s=0} implies \( \| r_j(t) \|^2 \lesssim \Psi(t) \), and differentiating yields
\[
\dot{\Psi}(t) = C \| r_j(t) \|^4 \lesssim \Psi(t)^2.
\]
This gives the differential inequality \( \dot{\Psi}(t) \lesssim \Psi(t)^2 \), which implies
\[
\frac{d}{dt} \left( \frac{1}{\Psi(t)} \right) \gtrsim -1.
\]
Integrating over \( [0, t] \subset [0, T\tau] \), we obtain
\[
\Psi(t) \le \frac{\varepsilon_{\tau}}{1 - C \varepsilon_{\tau} t}, \quad \text{for some constant } C > 0.
\]
Therefore, choosing \( \tau_0 \in (0, 1) \) sufficiently small such that \( C \varepsilon_{\tau} T\tau < 1/2 \) for all \( \tau \in (0, \tau_0) \), we conclude
\begin{align} \label{L^2_bound_of_r}
	\| r_j(t) \|^2 \leq 2 \varepsilon_{\tau}, \quad \text{for all } t \in [0, T\tau].
\end{align}

\underline{Case \( j = 1, \ s = 1 \)}. Multiplying \eqref{equn_r} with $\mathcal{I}_{1, 1} := -r_1^2 - 2 \partial_x^{2} r_1,$ it follows that
\begin{align*}
	&\langle \dot{r}_1, \mathcal{I}_{1, 1} \rangle = - \frac{1}{3} \frac{d}{dt} \int_{\mathbb{T}} r_1^3 dx + \frac{d}{dt} \| \partial_x r_1 \|^2,\\
	&\langle \mathcal{L}_1 r_1, \mathcal{I}_{1, 1} \rangle = 2  \int_{\mathbb{T}} r_1\,  \partial_x r_1\,  \partial_x^{2} r_1\,  dx,\\
	&\langle r_1 \partial_x r_1, \mathcal{I}_{1, 1} \rangle = - 2  \int_{\mathbb{T}} r_1\,  \partial_x r_1\,  \partial_x^{2} r_1\,  dx,\\
	&|\langle \mathcal{Q}(r_1, v_{\tau} + w_{\tau}), \mathcal{I}_{1, 1} \rangle| = \Big|- \frac{2}{3} \int_{\mathbb{T}}r_1^3 \partial_x(v_{\tau} + w_{\tau})  \ dx  +  \int_{\mathbb{T}} 3 \partial_x(v_{\tau} + w_{\tau})  (\partial_xr_1)^2 dx - \int_{\mathbb{T}} r_1^2 \partial_x^3(v_{\tau} + w_{\tau})  \ dx \Big| \\
	& \hspace{3cm} \le C \Big(\|(v_{\tau} + w_{\tau})\|_1 \|r_1\|^2 \|\partial_x r_1\| +  \|(v_{\tau} + w_{\tau})\|_{2} \|\partial_x r_1\|^2  + \| v_{\tau} + w_{\tau}\|_4 \| r_1\|^2\Big),\\
	& |\langle \mathcal{L}_1 r_1, \mathcal{I}_{1, 1}\rangle| \le C\Big(\|r_1\|^2 \| v_{\tau}\|_{4} + \|\partial_x r_1\| \|v_{\tau}\|_{4}\Big),\\
	& |\langle \mathcal{B}(v_{\tau}), \mathcal{I}_{1, 1}\rangle| \le C \Big(\| v_{\tau}\|_2^2 \| r_1\|^2 + \| v_{\tau}\|_{3}^2 \|r_1\|\Big).
\end{align*}
Gathering the above estimates and using \eqref{L^2_bound_of_r} and the fact that \(\varepsilon_\tau < 1\), hence \(\varepsilon_\tau^p < \varepsilon_\tau\) for any \(p > 1\), we denote $C\varepsilon_{\tau}$ again by $\varepsilon_{\tau}$. Integrating the resulting estimate over the time interval $[0,t]$, we obtain
\begin{align*}
	\int_{0}^{t}\left(- \frac{1}{3} \frac{d}{dt} \int_{\mathbb{T}} r_1^3 \, dx + \frac{d}{dt} \| \partial_x r_1 \|^2\right) \ d\rho \le \varepsilon_{\tau} + C \int_{0}^{t}  (\| v_{\tau} + w_{\tau} \|_{2}) \|\partial_x r_1\|^2 \, d\rho + C \int_0^t \| \partial_xr_1 \|^4 \, d\rho, \quad \text{for all } t \in [0, T\tau].
\end{align*}
So using \eqref{boundary_of_r} and \eqref{L^2_bound_of_r} we have,
\begin{align*}
	\| \partial_x r_1(t) \|^2 
	&\le C \int_{\mathbb{T}} |r_1(t)|^3 \, dx + \varepsilon_{\tau} + C \int_{0}^{t}  (\| v_{\tau} + w_{\tau} \|_{2}) \|\partial_x r_1\|^2 \, d\rho + C \int_0^t \| \partial_xr_1 \|^4 \, d\rho, \\
	&\le C \|r_1(t)\|_{L^\infty} \|r_1(t)\|^2 + \varepsilon_{\tau} + C \int_{0}^{t}  (\| v_{\tau} + w_{\tau} \|_{2}) \|\partial_x r_1\|^2 \, d\rho + C \int_0^t \| \partial_x r_1 \|^4 \, d\rho, \\
	&\le \frac{1}{2}\| \partial_x r_1(t) \| + C ( \|r_1(t)\|^2 + \| r_1(t)\|^4) + \varepsilon_{\tau} + C \int_{0}^{t}  (\| v_{\tau} + w_{\tau} \|_{2}) \|\partial_x r_1\|^2 \, d\rho + C \int_0^t \| \partial_xr_1 \|^4 \, d\rho, \\
	& \le \frac{1}{2}\| \partial_x r_1(t) \| + \varepsilon_{\tau} + C \int_{0}^{t}  (\| v_{\tau} + w_{\tau} \|_{2}) \|\partial_x r_1\|^2 \, d\rho + C \int_0^t \| \partial_x r_1 \|^4 \, d\rho,
\end{align*}
which implies, upon absorbing the first term of the right-hand side into the left,
\begin{align}\label{combining_s1_j1}
	\| \partial_x r_1(t) \|^2 \le \varepsilon_{\tau} + C \int_{0}^{t}  (\| v_{\tau} + w_{\tau} \|_{2}) \|\partial_x r_1\|^2 \, d\rho + C \int_0^t \| \partial_x r_1 \|^4 \, d\rho, \quad \text{for all } t \in [0, T\tau].
\end{align}
Again by the similar analysis what we have discussed after \eqref{combining_s0_any_j}, it is clear that 
\begin{align*}
	\|  \partial_x r_1(t) \|^2 \lesssim \varepsilon_{\tau}, \ \text{ for } t \in [0, T\tau].
\end{align*}
Combing this inequality with \eqref{L^2_bound_of_r}, we can see that
\begin{align}\label{H^1_bdd_r_for_ge_s1_j1}
\| r_1(t)\|_1^2 \lesssim \varepsilon_{\tau},  \ \text{ for } t \in [0, T\tau].
\end{align}

\underline{Case \( j = 1, \ s = 2 \)}. Motivating by \cite{Chen23}, consider the multiplier  
\[
\mathcal{I}_{1, 2} := \frac{18}{5} \partial_x^4 r_1 + 3 (\partial_x r_1)^2 + 6 r_1 \partial_x^2 r_1 + r_1^3.
\]
Proceeding with an energy estimate analogous to the previous cases, first we obtain
\begin{align*}
	 \| \partial_x^2 r_1(t) \| \lesssim \varepsilon_{\tau}, \quad \text{for all } t \in [0, T\tau],
\end{align*}
and then using \eqref{L^2_bound_of_r},
\begin{align}\label{norm_r_j=1,s=2}
\| r_1(t)\|_2^2 \lesssim \varepsilon_{\tau},  \ \text{ for } t \in [0, T\tau].	
\end{align}

\underline{Case \( j = 2, \ s = 2 \)}. Multiplying \eqref{equn_r} with $\mathcal{I}_{2, 2} := -r_2^2 - 2 \partial_x^2 r_2 + 2 \partial_x^{4} r_2,$ it follows that

\begin{align*}
&\frac{d}{dt} \left(- \frac{1}{3} \int_{\mathbb{T}} r_2^3 \, dx + \| \partial_x r_2\|^2 + \| \partial_x^2 r_2\|^2 \right) \\
&\quad = 2 \int_{\mathbb{T}} r_2 \partial_x r_2 \partial_x^2 r_2 \, dx + \int_{\mathbb{T}} \Big( r_2^2 + 2 \partial_x^2 r_2 - 2 \partial_x^4 r_2 \Big) \partial_x^5 (v_{\tau}) \, dx  + \int_{\mathbb{T}} \Big( r_2^2 + 2 \partial_x^2 r_2 - 2 \partial_x^4 r_2 \Big) r_2 \partial_x(v_{\tau} + w_{\tau}) \, dx \\
&\qquad + \int_{\mathbb{T}} \Big( r_2^2 + 2 \partial_x^2 r_2 - 2 \partial_x^4 r_2 \Big) (v_{\tau} + w_{\tau}) \partial_x r_2 \, dx + \int_{\mathbb{T}} \Big( r_2^2 + 2 \partial_x^2 r_2 - 2 \partial_x^4 r_2 \Big) v_{\tau} \partial_x(v_{\tau}) \, dx  \\
&\quad =: \mathcal{J}_1 + \mathcal{J}_2 + \mathcal{J}_3 + \mathcal{J}_4 + \mathcal{J}_5 .
\end{align*}
We estimate \( |\mathcal{J}_i| \) as follows:
\[
\begin{aligned}
	|\mathcal{J}_1| &\lesssim \|r_2\|^{\frac{4}{3}} + \|\partial_x r_2\|^4 + \| \partial_x^2 r_2\|^4 \\
	|\mathcal{J}_2| &\lesssim \| \partial_x^5 v_{\tau}\|_{L^{\infty}} \|r_2\|^2 + \| v_{\tau} \|_5 \| \partial_x^2 r_2 \| + \| v_{\tau} \|_7 \| \partial_x^2 r_2 \|, \\
	|\mathcal{J}_3| &\lesssim \|v_{\tau} + w_{\tau}\|_1 \|r_2\|^2 \|\partial_x r_2\|+  \| v_{\tau} + w_{\tau}\|_4 \|r_2\| \| \partial_x^2 r_2\| + \| v_{\tau} + w_{\tau} \|_2 \| \partial_x^2 r_2 \|^2 + \| v_{\tau} + w_{\tau} \|_4 \| \partial_x r_2 \| \| \partial_x^2 r_2\|, \\
	|\mathcal{J}_4| &\lesssim \|r_2\| \| v_{\tau} + w_{\tau}\| \| \partial_x r_2\|^2 + \| v_{\tau} + w_{\tau}\|_{\infty} \| \partial_x r_2 \| \| \partial_x^2 r_2 \| + \| v_{\tau} + w_{\tau} \|_1 \| \partial_x^2 r_2 \|^2 + \|v_{\tau} + w_{\tau}\|_3 \| \partial_x^2 r_2\| \|\partial_x r_2\|, \\
	|\mathcal{J}_5| &\lesssim \| v_{\tau} + w_{\tau}\|_1^2 \|r_2\| \| \partial_x r_2\| + \| v_{\tau} + w_{\tau}\|_1^2 \| \partial_x^2 r_2\| + \| v_{\tau} + w_{\tau}\|_4^2 \| \partial_x^2 r_2\|.
\end{aligned}
\]
Collecting the above estimates, repeating the previous analysis for $v_{\tau}$ and $w_{\tau}$, using \eqref{L^2_bound_of_r}, and finally integrating the resulting estimate over the time interval $[0,t]$, we have
\begin{align*}
	\int_{0}^{t}\left(- \frac{1}{3} \frac{d}{dt} \int_{\mathbb{T}} r_2^3 \, dx + \frac{d}{dt} \Big(\| \partial_x r_2 \|^2 + \| \partial_x^2 r_2\|^2\Big)\right) \ d\rho \le \varepsilon_{\tau} + C \int_{0}^{t}  \| v_{\tau} + w_{\tau} \|_{4} \Big(\|\partial_x r_2\|^2 + \| \partial_x^2 r_2\|^2 \Big) \, d\rho \\
	+ C \int_0^t \Big(\| \partial_x r_2 \|^2 + \| \partial_x^2 r_2\|^2 \Big)^2 \, d\rho, \quad \text{for all } t \in [0, T\tau].
\end{align*}
By the pervious analysis, it is clear that 
\begin{align*}
	\|  \partial_x r_2(t) \|^2 + \| \partial_x^2 r_2\|^2  \lesssim \varepsilon_{\tau}, \ \text{ for } t \in [0, T\tau].
\end{align*}
Combining this with \eqref{L^2_bound_of_r}, we have
\begin{align}\label{H^2_bdd_r_for_ge_s2_j2}
\|r_2(t)\|_2^2 \lesssim \varepsilon_{\tau}, \ \text{ for } t \in [0, T\tau].
\end{align}

\underline{Case \( j = 2, \ s = 1 \)}. This case follows from \eqref{L^2_bound_of_r}, \eqref{H^2_bdd_r_for_ge_s2_j2} and by an nonlinear interpolation \cite{Rosier_10,Bona_Chen_Saut_2002, Tartar_72} argument.

\underline{Case \( j \in \{1, 2\}, \ s \ge 3 \)}. We prove by inductive. Multiplying equation \eqref{equn_r} by \( \partial_x^{2s} r_j \) and integrating over the spatial domain, we obtain:
\begin{align*}
	\frac{1}{2} \frac{d}{dt} \| \partial_x^s r_j\|^2 \lesssim \frac{1}{2} |\langle \partial_x^{s + 1}r_j^2, \partial_x^s r_j \rangle| +  |\langle \partial_x^{s + 1} \Big((v_{\tau} + w_{\tau}) r_j\Big), \partial_x^s r_j\rangle| + \| v_{\tau}\|_{s + 2j + 1} \| \partial_x^s r_j\| + \| v_{\tau}\|^2_{s + 1} \| \partial_x^s r_j\|.
\end{align*}
It is easy to verify that
\[
\left| \left\langle \partial_x^{s+1}(r_j^2),\ \partial_x^s r_j \right\rangle \right| \lesssim \sum_{m=0}^{s+1} \left| \int_{\mathbb{T}} \partial_x^m r_j\, \partial_x^{s+1-m} r_j\, \partial_x^s r_j\, dx \right|.
\]
We estimate each term on the right-hand side. For \( m = 0 \) and \( m = s+1 \), 
\[
\left| \int_{\mathbb{T}} r_j\, \partial_x^{s+1} r_j\, \partial_x^s r_j\, dx \right| = \frac{1}{2} \left| \int_{\mathbb{T}} \partial_x r_j\, (\partial_x^s r_j)^2\, dx \right| \leq C \|\partial_x r_j\|_{L^{\infty}} \|\partial_x^s r_j\|^2 \leq C \|r_j\|_{2} \|\partial_x^s r_j\|^2.
\]
Similarly, for \( m = 1 \) and \( m = s \), 
\[
\left| \int_{\mathbb{T}} \partial_x r_j\, (\partial_x^s r_j)^2\, dx \right| \leq \|\partial_x r_j\|_{L^{\infty}} \|\partial_x^s r_j\|^2 \leq C \|r_j\|_{2} \|\partial_x^s r_j\|^2.
\]
For \( 2 \leq m \leq s - 1 \),
\begin{align*}
	\left| \int_{\mathbb{T}} \partial_x^m r_j\, \partial_x^{s+1 - m} r_j\, \partial_x^s r_j\, dx \right| 
	& \leq \|\partial_x^m r_j\|_{L^\infty} \int_{\mathbb{T}} \left| \partial_x^{s+1 - m} r_j \ \partial_x^s r_j \right| dx \\
	& \lesssim \|r_j\|_{{m+1}} \|\partial_x^{s+1 - m} r_j\| \|\partial_x^s r_j\| \\
	& \lesssim \|r_j\|_{s} \|r_j\|_{{s-1}} \|\partial_x^s r_j\| \\
	& \lesssim \left( \|r_j\| + \|\partial_x^s r_j\| \right) \|r_j\|_{{s-1}} \|\partial_x^s r_j\|.
\end{align*}
Putting together all the estimates above, we arrive at
\[
\left| \left\langle \partial_x^{s+1}(r_j^2),\ \partial_x^s r_j \right\rangle \right| \lesssim \left( \|r_j\|_{2} \|\partial_x^s r_j\|^2 + \|r_j\| \|r_j\|_{{s-1}} \|\partial_x^s r_j\| + \|r_j\|_{{s-1}} \|\partial_x^s r_j\|^2 \right).
\]
Similarly, we can show:  
\[
\left| \left\langle  \partial_x^{s + 1} \Big((v_{\tau} + w_{\tau}) r_j\Big),\ \partial_x^s r_j \right\rangle \right| \lesssim \Big(\|(v_{\tau} + w_{\tau})\|_{s+2} \|r_j\| \|\partial_x^s r_j\| + \|(v_{\tau} + w_{\tau})\|_{s+2} \|\partial_x^s r_j\|^2\Big).
\]  
Combining this with the previous estimates,  repeating the previous analysis for $v_{\tau}$ and $w_{\tau}$, using \eqref{L^2_bound_of_r}, and finally integrating the resulting estimate over the time interval $[0,t]$, we arrive at
\begin{align*}
	\| \partial_x^s r_j(t) \|^2 \le \varepsilon_{\tau} + C \int_{0}^{t}  (\| v_{\tau} + w_{\tau} \|_{s + 2}) \|\partial_x^s r_j\|^2 \, d\rho + C \int_0^t \|\partial_x^s r_j \|^4 \, d\rho, \quad \text{for all } t \in [0, T\tau].
\end{align*}
By the Gronwall inequality and same type of analysis done after \eqref{combining_s0_any_j}, 
\begin{align*}
\| \partial_x^s r_j (t)  \|^2 \lesssim \varepsilon_{\tau}, \ \text{ for } t \in [0, T\tau].
\end{align*}
Combining this with \eqref{L^2_bound_of_r}, we have
\begin{align}\label{H^s_bdd_r_for_s_ge3_any_j}
	\|r_j(t)\|_s^2 \lesssim \varepsilon_{\tau}, \ \text{ for } t \in [0, T\tau].
\end{align}
This completes the proof of the proposition.
\end{proof}

\section{Proof of the main result}
\label{section-aaprox_control}
In this section, we prove the main controllability result. By directly combining the approximate controllability of \eqref{equ_30} (cf.~Proposition~\ref{prpn_approx_linear_equn}) with the closeness of the solutions to the linear equation \eqref{linearized_Burger} and the nonlinear equation \eqref{ctrl_prblm_without_initial_fr_proof} (cf.~Proposition~\ref{prpn_closeness_solun}), one can obtain the approximate controllability of \eqref{ctrl_prblm} using a finite-dimensional control. However, proceeding in this way does not allow us to express the control in the form \eqref{form_of_ctrl}. To achieve this, we need the notion of an \textit{approximate right inverse} for a linear operator.

 Let \( F \) and \( H \) be separable Hilbert spaces, and let \( \mathfrak{T}: F \to H \) be a bounded linear operator. In general, the equation \(  \mathfrak{T} f = h \), may either admit no solution or possess multiple solutions. However, under suitable assumptions, it is possible to construct an \textit{approximate} solution that depends linearly on the given data \( h \). The results below formalizes this idea.
\begin{proposition}\cite[Proposition 2.6]{Kuksin_Nersesyan_Shirikyan_2020}\label{prp_approx_right_inverse}
Let \(  \mathfrak{T}: F \to H \) be a continuous linear operator between separable Hilbert spaces, and assume that \( \mathrm{Im}(\mathfrak{T}) \) is dense in \( H \). Let \( V \) be a Banach space compactly embedded in \( H \). Then, for any \( \varepsilon > 0 \), there exists a continuous linear operator \(  \mathfrak{T}_{\varepsilon}: H \to F \) with finite-dimensional range such that
\begin{align}\label{equn_approx_right_inverse}
\|  \mathfrak{T} ( \mathfrak{T}_{\varepsilon} h) - h \|_{H} \leq \varepsilon \|h\|_{V}, \quad \text{for all } h \in V.
\end{align}
\end{proposition}

\begin{proof}[Proof of Theorem~\ref{main_thm}]
Given \( s \in \mathbb{N} \), we define the resolving operator
	\[
	\mathfrak{R} : H^{s}_0(\mathbb{T}) \times L^2((0, T); \mathcal{H}) \to C([0, T]; H^{s}_0(\mathbb{T})), \quad (v^0, g) \mapsto v,
	\]  
	which maps the initial datum and forcing term to the solution \( v \) of the linearized inviscous Burgers equation \eqref{linearized_Burger} with initial condition \( v(0) = v^0 \). We denote by \( \mathfrak{R}_t \) its evaluation at time \( t \in [0, T] \). 
	From Proposition~\ref{prpn_approx_linear_equn}, we have that the image of the operator  
	\[
	\mathfrak{R}_T(0, \cdot) : L^2((0, T); \mathcal{H}) \to H^s_0(\mathbb{T})
	\]  
	is dense in \( H^s_0(\mathbb{T}) \), i.e.,  
	\[
	\operatorname{Im}\left( \mathfrak{R}_T(0, \cdot) \right) \text{ is dense in } H^s_0(\mathbb{T}).
	\] 
Essentially, this operator $\mathfrak{R}_T(0, \cdot)$ same as $\mathcal{A}_1$ is the proof of Proposition~\eqref{prpn_approx_linear_equn}. We apply Proposition~\ref{prp_approx_right_inverse} with the following identification of spaces:
	\[
	\mathfrak{T} := \mathfrak{R}_T(0, \cdot), \quad H := H^s_0(\mathbb{T}), \quad F := L^2((0, T); \mathcal{H}), \quad \text{and} \quad V := H^{s + 2j + 1}_0(\mathbb{T}).
	\]  
	Then, for any \( \varepsilon > 0 \), there exists a continuous linear operator by Proposition \ref{prp_approx_right_inverse}
	\[
	\mathfrak{T}_{\varepsilon} : H^s_0(\mathbb{T}) \to L^2((0, T); \mathcal{H}),
	\]  
	with finite-dimensional range, such that for all \( h \in H^{s + 2j + 1}_0(\mathbb{T}) \), we have  
	\[
	\left\| \mathfrak{R}_T(0,  \mathfrak{T}_\varepsilon h) - h \right\|_{H^s} \leq \varepsilon \| h \|_{H^{s + 2j + 1}}.
	\]  
	Choose initial and target states \( u^0, u^1 \in H^{s + 2j + 1}_0(\mathbb{T}) \). Then there exists a constant \( M > 0 \) such that  
	\[
	u^0, u^1 \in B_{H^{s + 2j + 1}_0(\mathbb{T})}(0, M),
	\]  
	the ball of radius \( M \) centered at the origin in \( H^{s + 2j + 1}_0(\mathbb{T}) \).  
	We now apply the approximate right inverse estimate established earlier with  
	\[
	h := u_1 - \mathfrak{R}_T(u^0, 0).
	\]  
	Let \( \theta > 0 \), and define
	\[
	g_\theta :=  \mathfrak{T}_\theta\left( u^1 - \mathfrak{R}_T(u^0, 0) \right).
	\]  
	Then, by construction, we have  
	\[
	\left\| \mathfrak{R}_T(0, g_\theta) - \left( u^1 - \mathfrak{R}_T(u^0, 0) \right) \right\|_{s} \leq \theta \left\| u^1 - \mathfrak{R}_T(u^0, 0) \right\|_{{s + 2j + 1}}.
	\]  
	Using the linearity for the equation \eqref{linearized_Burger}, as well as the associated resolving operator \( \mathfrak{R}_T \), we deduce  
	\[
	\mathfrak{R}_T(u^0, g_\theta) = \mathfrak{R}_T(u^0, 0) + \mathfrak{R}_T(0, g_\theta),
	\]  
	so for some constant \( C = C(M) > 0 \) depending only on \( M \). That is,  
	\begin{align} \label{ineq_1}
		\left\| \mathfrak{R}_T(u^0, g_\theta) - u^1 \right\|_{s} \leq \theta C.
	\end{align}

	We now apply Proposition~\ref{prpn_closeness_solun} with the previously chosen initial data \( u^0 \in H^{s + 2j + 1}_0(\mathbb{T}) \) and control \( g := g_{\theta} \in L^2((0, T); \mathcal{H}) \). Then, by the convergence result \eqref{lim_closness_solun}, for any \( \delta > 0 \), there exists \( \tau_1 > 0 \) and a control  
	\[
	\eta_{\tau_1, j} \in L^2((0, T\tau_1); \mathcal{H})
	\]  
	such that  
	\begin{align} \label{ineq_2}
		\left\| \mathcal{R}_{j, T\tau_1}(u_0, \eta_{\tau_1, j}) - v(T) \right\|_{H^s} \leq \delta,
	\end{align}  
	here \( v(T) = \mathfrak{R}_T(u^0, g_\theta) \). Combining \eqref{ineq_1} and \eqref{ineq_2}, we obtain the estimate  
	\begin{align} \label{fnl_ineq}
		\left\| \mathcal{R}_{j, T\tau_1}(u^0, \eta_{\tau_1, j}) - u^1 \right\|_{s}
		&\leq \left\| \mathcal{R}_{j, T\tau_1}(u^0, \eta_{\tau_1, j}) - v(T) \right\|_{s}
		+ \left\| \mathfrak{R}_T(u^0, g_\theta) - u^1 \right\|_{s} \\
		&\leq \delta + \theta C. \nonumber
	\end{align}
	
	Since both \( \delta > 0 \) and \( \theta > 0 \) are arbitrary, the right-hand side of \eqref{fnl_ineq} can be made arbitrarily small. Therefore, for sufficiently small \( \tau > 0 \), there exists a control  
	\[
	\eta_{\tau, j} \in L^2((0, T\tau); \mathcal{H})
	\]  
	such that  
	\[
	\mathcal{R}_{j, T\tau}(u^0, \eta_{\tau, j}) \to u^1 \quad \text{in } H^s \quad \text{as } \tau \to 0^+.
	\]  
	This completes the proof of the limit \eqref{lim_u_Ttau}, and the convergence in  is uniform for all \( u^0, u^1 \) in bounded subsets of \( H^{s + 2j + 1}_0(\mathbb{T}) \). Now, for the structure of the control, we use \eqref{form_ctrl_as_solun_linear_equn} with \( g = g_{\theta} \), and defining $\kappa_{\tau, j} (t):= \xi_{\tau}(t) + \mathcal{L}_jw(\tau^{-1}t)$, we obtain \eqref{form_of_ctrl}.
\end{proof}

As a immediate consequence, we get the Corollary~\ref{Coro_to_cnclude_main_result}.
\begin{proof}[Proof of Corollary~\ref{Coro_to_cnclude_main_result}]
	By the density of \( H^{s + 2j + 1}_0(\mathbb{T}) \) in \( H^s_0(\mathbb{T}) \) with semi-global stability \eqref{stability}, Theorem \ref{main_thm} ensures that for any \( u^0, u^1 \in H^s_0(\mathbb{T}) \), there exists a control \( \tilde{\eta}_{\tau, j} \in L^2((0, T\tau); \mathcal{H}) \) such that  
	\[
	\mathcal{R}_{j, T\tau}(u^0, \tilde{\eta}_{\tau, j}) \to u^1 \quad \text{in } H^s \text{ as } \tau \to 0^+.
	\]
	
	Now, by continuity of the flow in time and stability with respect to initial data (Propositions \ref{prp_existence}–\ref{prp_stability}), there exist \( \mathfrak{r} > 0 \), \( t' > 0 \) such that any trajectory starting in \( B(u^1, \mathfrak{r}) \) remains \( \varepsilon \)-close to \( u^1 \) for time \( t' \).
	
	Choosing \( \tau_1 \) small so that \( \mathcal{R}_{j, T\tau_1}(u^0, \tilde{\eta}_{\tau_1, j}) \in B(u^1, \mathfrak{r}) \), we patch the control with zero beyond \( T\tau_1 \) to extend the trajectory. If \( T\tau_1 + t' \geq T \), we are done. Otherwise, we iterate this argument finitely many times until reaching time \( T \). Thus, the system is approximately controllable at time \( T \) in \( H^s_0(\mathbb{T}) \).
\end{proof}

\begin{remark}
	The control \( \eta \) obtained in this corollary does not maintain the structure described in \eqref{form_of_ctrl} throughout the full time interval. In particular, the dependence on \( (u^0, u^1) \) is no longer affine once the zero control is applied in a neighborhood of \( u^1 \). This loss of structure is a consequence of the nonlinearity in the map \( \mathcal{R}_{j, t}(u^1, 0) \) with respect to \( u^1 \). Nevertheless, for any fixed \( \varepsilon, M > 0 \) and for initial and target states \( u^0, u^1 \in B_{H^{s + 2j + 1}_0(\mathbb{T})}(0, M) \), the part of the control on the interval \( [0, T\tau] \) still matches the form given in \eqref{form_of_ctrl}, while the portion on \( [T\tau, T] \) does not depend on \( u^0 \).
\end{remark}

\bigskip

\bigskip

\noindent \textbf{Data Availability Statement} This study is purely theoretical, and does not rely on any specific dataset. Therefore, data sharing is not applicable.


\noindent \textbf{Declaration}

\noindent \textbf{Funding} This work was supported by the Integrated PhD Fellowship at IISER Kolkata, India.

\noindent \textbf{Conflict of interest} The author declares that there are no conflicts of interest.

\bibliographystyle{alpha}
\bibliography{reference}

\newpage

\end{document}

%% file: jitz.tikzstyles
\tikzstyle{Dotty}=[fill=black, draw=black, shape=circle, tikzit shape=circle]
\tikzstyle{gyyu}=[fill=cyan, draw=black, shape=circle, tikzit shape=circle]
\tikzstyle{doty}=[fill=black, draw=black, shape=circle, tikzit fill=black, tikzit draw=black, tikzit shape=circle]

\tikzstyle{eddt}=[fill=none, draw=black, tikzit draw=black, ->]
\tikzstyle{new edge style 0}=[-, draw=red, tikzit draw=red]
\tikzstyle{ArrowBlack}=[draw=none, tikzit draw=none, ->, fill=white]

%% file: higher__DISPERSIV_linear__test2.bbl
\newcommand{\etalchar}[1]{$^{#1}$}
\begin{thebibliography}{dACKV22}

\bibitem[ACFD12]{Araruna_Capistrano_Doronin_2012}
F.~D. Araruna, R.~A. Capistrano-Filho, and G.~G. Doronin.
\newblock Energy decay for the modified {K}awahara equation posed in a bounded
  domain.
\newblock {\em J. Math. Anal. Appl.}, 385(2):743--756, 2012.

\bibitem[ADR26]{Araruna_Doronin_Rosier_2026}
F\'{a}gner~D. Araruna, Gleb~G. Doronin, and Lionel Rosier.
\newblock On the critical lengths and controllability of the kawahara equation.
\newblock {\em Communications in Contemporary Mathematics}, (265000):27, 2026.

\bibitem[AM25]{Ahamed_Mondal_25}
Sakil Ahamed and Debanjit Mondal.
\newblock Global controllability of the kawahara equation at any time.
\newblock {\em Nonlinear Analysis: Real World Applications}, 85:104374, 2025.

\bibitem[AS05]{Agrachev_2005}
Andrey~A. Agrachev and Andrey~V. Sarychev.
\newblock Navier-{Stokes} equations: controllability by means of low modes
  forcing.
\newblock {\em J. Math. Fluid Mech.}, 7(1):108--152, 2005.

\bibitem[AS06]{Agrachev_2006}
Andrey~A. Agrachev and Andrey~V. Sarychev.
\newblock Controllability of 2d {Euler} and {Navier}-{Stokes} equations by
  degenerate forcing.
\newblock {\em Commun. Math. Phys.}, 265(3):673--697, 2006.

\bibitem[BCS02]{Bona_Chen_Saut_2002}
J.~L. Bona, M.~Chen, and J.-C. Saut.
\newblock Boussinesq equations and other systems for small-amplitude long waves
  in nonlinear dispersive media. {I}: {Derivation} and linear theory.
\newblock {\em J. Nonlinear Sci.}, 12(4):283--318, 2002.

\bibitem[Bou93a]{Bourgain_1997}
J.~Bourgain.
\newblock Fourier transform restriction phenomena for certain lattice subsets
  and applications to nonlinear evolution equations. {I}: {Schr{\"o}dinger}
  equations.
\newblock {\em Geom. Funct. Anal.}, 3(2):107--156, 1993.

\bibitem[Bou93b]{Bourgain_1993}
J.~Bourgain.
\newblock Fourier transform restriction phenomena for certain lattice subsets
  and applications to nonlinear evolution equations. {II}: {The}
  {KdV}-equation.
\newblock {\em Geom. Funct. Anal.}, 3(3):209--262, 1993.

\bibitem[CC04]{Coron_2004}
Jean-Michel Coron and Emmanuelle Cr\'{e}peau.
\newblock Exact boundary controllability of a nonlinear {K}d{V} equation with
  critical lengths.
\newblock {\em J. Eur. Math. Soc. (JEMS)}, 6(3):367--398, 2004.

\bibitem[CC09]{Cerpa_Crepeau_2009}
Eduardo Cerpa and Emmanuelle Cr\'epeau.
\newblock Boundary controllability for the nonlinear {K}orteweg-de {V}ries
  equation on any critical domain.
\newblock {\em Ann. Inst. H. Poincar\'e{} C Anal. Non Lin\'eaire},
  26(2):457--475, 2009.

\bibitem[CDM25]{Chowdhury_Dutta_Mondal_25}
Shirshendu Chowdhury, Rajib Dutta, and Debanjit Mondal.
\newblock Global approximate controllability of the {C}amassa--{H}olm equation
  by a finite dimensional force.
\newblock {\em ESAIM Control Optim. Calc. Var.}, 31:Paper No. 86, 2025.

\bibitem[Cer07]{Cerpa_2007}
Eduardo Cerpa.
\newblock Exact controllability of a nonlinear {K}orteweg-de {V}ries equation
  on a critical spatial domain.
\newblock {\em SIAM J. Control Optim.}, 46(3):877--899, 2007.

\bibitem[Cer14]{Cerpa_2014}
Eduardo Cerpa.
\newblock Control of a {K}orteweg-de {V}ries equation: a tutorial.
\newblock {\em Math. Control Relat. Fields}, 4(1):45--99, 2014.

\bibitem[CF96]{Coron_Fursikov_1996}
Jean-Michel Coron and Andrei~V. Fursikov.
\newblock Global exact controllability of the 2d {Navier}-{Stokes} equations on
  a manifold without boundary.
\newblock {\em Russ. J. Math. Phys.}, 4(4):429--448, 1996.

\bibitem[CFPR15]{Capistrano_Roberto_2015}
Roberto~A. Capistrano-Filho, Ademir~F. Pazoto, and Lionel Rosier.
\newblock Internal controllability of the {K}orteweg--de {V}ries equation on a
  bounded domain.
\newblock {\em ESAIM Control Optim. Calc. Var.}, 21(4):1076--1107, 2015.

\bibitem[Cha09a]{Chapouly_KdV_2009}
Marianne Chapouly.
\newblock Global controllability of a nonlinear {Korteweg}-de {Vries} equation.
\newblock {\em Commun. Contemp. Math.}, 11(3):495--521, 2009.

\bibitem[Cha09b]{Chapouly_Burger_2009}
Marianne Chapouly.
\newblock Global controllability of nonviscous and viscous {Burgers}-type
  equations.
\newblock {\em SIAM J. Control Optim.}, 48(3):1567--1599, 2009.

\bibitem[Che23]{Chen23}
Mo~Chen.
\newblock Global approximate controllability of the {K}orteweg--de {V}ries
  equation by a finite-dimensional force.
\newblock {\em Appl. Math. Optim.}, 87(1):Paper No. 12, 22, 2023.

\bibitem[CKN24]{Coron_Koenig_Nguyen_24}
Jean-Michel Coron, Armand Koenig, and Hoai-Minh Nguyen.
\newblock On the small-time local controllability of a {K}d{V} system for
  critical lengths.
\newblock {\em J. Eur. Math. Soc. (JEMS)}, 26(4):1193--1253, 2024.

\bibitem[CKS{\etalchar{+}}03]{Coliander_Keel_Staffilani_Tao_2003}
J.~Colliander, M.~Keel, G.~Staffilani, H.~Takaoka, and T.~Tao.
\newblock Sharp global well-posedness for {K}d{V} and modified {K}d{V} on
  {$\Bbb R$} and {$\Bbb T$}.
\newblock {\em J. Amer. Math. Soc.}, 16(3):705--749, 2003.

\bibitem[Cor92]{Coron_return_stabilization_1992}
Jean-Michel Coron.
\newblock Global asymptotic stabilization for controllable systems without
  drift.
\newblock {\em Math. Control Signals Syst.}, 5(3):295--312, 1992.

\bibitem[Cor96a]{Coron_JMPA_1996}
J.-M. Coron.
\newblock On the controllability of 2-{D} incompressible perfect fluids.
\newblock {\em J. Math. Pures Appl. (9)}, 75(2):155--188, 1996.

\bibitem[Cor96b]{Coron_ESIAM_1996}
Jean-Michel Coron.
\newblock On the controllability of the 2-{D} incompressible {Navier}-{Stokes}
  equations with the {Navier} slip boundary conditions.
\newblock {\em ESAIM, Control Optim. Calc. Var.}, 1:35--75, 1996.

\bibitem[Cor07]{Coron_book_2007}
Jean-Michel Coron.
\newblock {\em Control and nonlinearity}, volume 136 of {\em Mathematical
  Surveys and Monographs}.
\newblock American Mathematical Society, Providence, RI, 2007.

\bibitem[CXZ23]{Coron_Xiang_Zhang_2023}
Jean-Michel Coron, Shengquan Xiang, and Ping Zhang.
\newblock On the global approximate controllability in small time of
  semiclassical 1-{D} {S}chr\"odinger equations between two states with
  positive quantum densities.
\newblock {\em J. Differential Equations}, 345:1--44, 2023.

\bibitem[dACKV22]{Roberto_Kwak_Leal_2022}
Roberto de~A.~Capistrano–Filho, Chulkwang Kwak, and Francisco~J. {Vielma
  Leal}.
\newblock On the control issues for higher-order nonlinear dispersive equations
  on the circle.
\newblock {\em Nonlinear Analysis: Real World Applications}, 68:103695, 2022.

\bibitem[FI99]{Fursikov_Imanuvilov_1999}
A.~V. Fursikov and O.~Yu. Imanuvilov.
\newblock Exact controllability of the {Navier}-{Stokes} and {Boussinesq}
  equations.
\newblock {\em Russ. Math. Surv.}, 54(3):565--618, 1999.

\bibitem[Gao22]{Peng_Gao_2022}
Peng Gao.
\newblock Irreducibility of {K}uramoto-{S}ivashinsky equation driven by
  degenerate noise.
\newblock {\em ESAIM Control Optim. Calc. Var.}, 28:Paper No. 20, 22, 2022.

\bibitem[GG09]{Glass_Guerrero_2009}
O.~Glass and S.~Guerrero.
\newblock On the controllability of the fifth-order {K}orteweg-de {V}ries
  equation.
\newblock {\em Ann. Inst. H. Poincar\'{e} C Anal. Non Lin\'{e}aire},
  26(6):2181--2209, 2009.

\bibitem[Gla00]{Glass_2000}
Olivier Glass.
\newblock Exact boundary controllability of 3-{D} {Euler} equation.
\newblock {\em ESAIM, Control Optim. Calc. Var.}, 5:1--44, 2000.

\bibitem[Jel23]{Jellouli_23}
Melek Jellouli.
\newblock On the controllability of the {BBM} equation.
\newblock {\em Math. Control Relat. Fields}, 13(1):415--430, 2023.

\bibitem[Kat12]{Kato_2012}
Takamori Kato.
\newblock Low regularity well-posedness for the periodic {Kawahara} equation.
\newblock {\em Differ. Integral Equ.}, 25(11-12):1011--1036, 2012.

\bibitem[Kaw72]{kawahara_1972}
Takuji Kawahara.
\newblock Oscillatory solitary waves in dispersive media.
\newblock {\em Journal of the physical society of Japan}, 33(1):260--264, 1972.

\bibitem[KDV95]{KdV_1895}
D.~J. Korteweg and G.~De~Vries.
\newblock On the change of form of long waves advancing in a rectangular canal,
  and on a new type of long stationary waves.
\newblock {\em Phil. Mag. (5)}, 39:422--443, 1895.

\bibitem[KNS20]{Kuksin_Nersesyan_Shirikyan_2020}
Sergei Kuksin, Vahagn Nersesyan, and Armen Shirikyan.
\newblock Exponential mixing for a class of dissipative {PDEs} with bounded
  degenerate noise.
\newblock {\em Geom. Funct. Anal.}, 30(1):126--187, 2020.

\bibitem[KPV94]{Kenig_Ponce_gustavo_1994}
Carlos~E. Kenig, Gustavo Ponce, and Luis Vega.
\newblock Higher-order nonlinear dispersive equations.
\newblock {\em Proc. Am. Math. Soc.}, 122(1):157--166, 1994.

\bibitem[KPV96]{Kenig_Carlos_Ponce_1996}
Carlos~E. Kenig, Gustavo Ponce, and Luis Vega.
\newblock A bilinear estimate with applications to the {KdV} equation.
\newblock {\em J. Am. Math. Soc.}, 9(2):573--603, 1996.

\bibitem[LRZ10]{Rosier_10}
Camille Laurent, Lionel Rosier, and Bing-Yu Zhang.
\newblock Control and stabilization of the {K}orteweg-de {V}ries equation on a
  periodic domain.
\newblock {\em Comm. Partial Differential Equations}, 35(4):707--744, 2010.

\bibitem[Ner10]{Nersisyan_2010}
Hayk Nersisyan.
\newblock Controllability of 3{D} incompressible {E}uler equations by a
  finite-dimensional external force.
\newblock {\em ESAIM Control Optim. Calc. Var.}, 16(3):677--694, 2010.

\bibitem[Ner11]{Nersisyan_2011}
Hayk Nersisyan.
\newblock Controllability of the 3{D} compressible {E}uler system.
\newblock {\em Comm. Partial Differential Equations}, 36(9):1544--1564, 2011.

\bibitem[Ner21a]{Narsesyan_21}
Vahagn Nersesyan.
\newblock Approximate controllability of nonlinear parabolic {PDE}s in
  arbitrary space dimension.
\newblock {\em Math. Control Relat. Fields}, 11(2):237--251, 2021.

\bibitem[Ner21b]{Nersesyan_2021_Sicon}
Vahagn Nersesyan.
\newblock A proof of approximate controllability of the 3d {Navier}-{Stokes}
  system via a linear test.
\newblock {\em SIAM J. Control Optim.}, 59(4):2411--2427, 2021.

\bibitem[Ngu25]{Nguyen_2025}
Hoai-Minh Nguyen.
\newblock Local controllability of the {K}orteweg--de {V}ries equation with the
  right {D}irichlet control.
\newblock {\em J. Differential Equations}, 435:Paper No. 113235, 85, 2025.

\bibitem[Ros97]{Rosier_1997}
Lionel Rosier.
\newblock Exact boundary controllability for the {K}orteweg-de {V}ries equation
  on a bounded domain.
\newblock {\em ESAIM Control Optim. Calc. Var.}, 2:33--55, 1997.

\bibitem[Ros04]{Rosier_2004}
Lionel Rosier.
\newblock Control of the surface of a fluid by a wavemaker.
\newblock {\em ESAIM Control Optim. Calc. Var.}, 10(3):346--380, 2004.

\bibitem[Rud76]{Rudin_analysis_book}
Walter Rudin.
\newblock {\em Principles of mathematical analysis}.
\newblock International Series in Pure and Applied Mathematics. McGraw-Hill
  Book Co., New York-Auckland-D\"usseldorf, third edition, 1976.

\bibitem[Rus91]{Russell_1991}
David~L. Russell.
\newblock Computational study of the {K}orteweg-de {V}ries equation with
  localized control action.
\newblock In {\em Distributed parameter control systems ({M}inneapolis, {MN},
  1989)}, volume 128 of {\em Lecture Notes in Pure and Appl. Math.}, pages
  195--203. Dekker, New York, 1991.

\bibitem[RZ93]{Russell_Zhang_1993}
D.~L. Russell and Bing~Yu Zhang.
\newblock Controllability and stabilizability of the third-order linear
  dispersion equation on a periodic domain.
\newblock {\em SIAM J. Control Optim.}, 31(3):659--676, 1993.

\bibitem[RZ96]{Russell_Zang_1996}
David~L. Russell and Bing~Yu Zhang.
\newblock Exact controllability and stabilizability of the {K}orteweg-de
  {V}ries equation.
\newblock {\em Trans. Amer. Math. Soc.}, 348(9):3643--3672, 1996.

\bibitem[Sar12]{Sarychev_12}
Andrey Sarychev.
\newblock Controllability of the cubic {S}chroedinger equation via a
  low-dimensional source term.
\newblock {\em Math. Control Relat. Fields}, 2(3):247--270, 2012.

\bibitem[Shi06]{Shirikyan_06}
Armen Shirikyan.
\newblock Approximate controllability of three-dimensional {N}avier-{S}tokes
  equations.
\newblock {\em Comm. Math. Phys.}, 266(1):123--151, 2006.

\bibitem[Shi07]{Shirikyan_07}
Armen Shirikyan.
\newblock Contr\^{o}labilit\'{e} exacte en projections pour les \'{e}quations
  de {N}avier-{S}tokes tridimensionnelles.
\newblock {\em Ann. Inst. H. Poincar\'{e} C Anal. Non Lin\'{e}aire},
  24(4):521--537, 2007.

\bibitem[Shi14]{Shirikyan_14}
Armen Shirikyan.
\newblock Approximate controllability of the viscous {B}urgers equation on the
  real line.
\newblock In {\em Geometric control theory and sub-{R}iemannian geometry},
  volume~5 of {\em Springer INdAM Ser.}, pages 351--370. Springer, Cham, 2014.

\bibitem[Shi18]{Shirikyan_18}
Armen Shirikyan.
\newblock Control theory for the {B}urgers equation: {A}grachev-{S}arychev
  approach.
\newblock {\em Pure Appl. Funct. Anal.}, 3(1):219--240, 2018.

\bibitem[Tao01]{Tao_2001}
Terence Tao.
\newblock Multilinear weighted convolution of {{\(L^2\)}} functions, and
  applications to nonlinear dispersive equations.
\newblock {\em Am. J. Math.}, 123(5):839--908, 2001.

\bibitem[Tar72]{Tartar_72}
L.~Tartar.
\newblock Interpolation non lin\'eaire et r\'egularit\'e.
\newblock {\em J. Functional Analysis}, 9:469--489, 1972.

\bibitem[Zha90]{Zhang_1990_thesis}
Bingyu Zhang.
\newblock {\em Some results for nonlinear dispersive wave equations with
  applications to control}.
\newblock ProQuest LLC, Ann Arbor, MI, 1990.
\newblock Thesis (Ph.D.)--The University of Wisconsin - Madison.

\bibitem[ZZ12]{Zhang_Zhao_2012}
Bing-Yu Zhang and Xiangqing Zhao.
\newblock Control and stabilization of the {K}awahara equation on a periodic
  domain.
\newblock {\em Commun. Inf. Syst.}, 12(1):77--95, 2012.

\bibitem[ZZ15]{Zhang_Zhao_2015}
Xiangqing Zhao and Bing-Yu Zhang.
\newblock Global controllability and stabilizability of {K}awahara equation on
  a periodic domain.
\newblock {\em Math. Control Relat. Fields}, 5(2):335--358, 2015.

\end{thebibliography}
